\documentclass{amsart}
\usepackage[utf8]{inputenc}
\usepackage{amsmath,amsfonts,amssymb, dsfont}
\usepackage[scale=.7]{geometry}
\usepackage{xcolor}
\usepackage{tikz}
\usepackage{pgfplots}
\pgfkeys{/pgf/declare function={asinh(\x) = ln(\x + sqrt((\x)^2+1));}}
\usetikzlibrary{quotes,angles, calc}
\tikzset{>=latex}
\tikzstyle{vector}=[->,very thick]
\tikzstyle{mydashed}=[dash pattern=on 2pt off 2pt]
\newcommand{\NN}{\mathbb{N}}

\newcommand{\RR}{\mathbb{R}}
\newcommand{\cE}{\mathcal{E}}
\newcommand{\cX}{\mathcal{X}}
\newcommand{\cM}{\mathcal{M}}
\newcommand{\cH}{\mathcal{H}}
\newcommand{\cT}{\mathcal{T}}

\newcommand{\dd}{\mathrm{d}}

\newcommand{\Id}{\mathrm{Id}}

\newcommand{\pa}{\partial}
\newcommand{\eps}{\varepsilon} 
\newcommand{\ds}{\displaystyle}

\newcommand{\blue}[1]{#1}

\newtheorem{theo}{Theorem}[section]
\newtheorem{prop}[theo]{Proposition}
\newtheorem{cor}[theo]{Corollary}
\newtheorem{lem}[theo]{Lemma}
\newtheorem{rema}[theo]{Remark}
\newtheorem{defi}[theo]{Definition}

\title{On a Vlasov-Fokker-Planck equation for stored electron beams}
\date{\today}

\author[Ludovic Cesbron]{Ludovic Cesbron}
\address[Ludovic Cesbron]{CY Cergy Paris Universit\'e, Laboratoire Analyse, G\'eom\'etrie et Mod\'elisation (UMR CNRS 8088), 2 avenue Adolphe-Chauvin BP 222, 95302 Pontoise cedex}
\email{ludovic.cesbron@cyu.fr}
\author[Maxime Herda]{Maxime Herda}
\address[Maxime Herda]{Inria, Universit\'e de Lille, CNRS, UMR 8524 - Laboratoire Paul Painlev\'e,
F-59000 Lille, France}
\email{maxime.herda@inria.fr}

\begin{document}
\begin{abstract}
In this paper we study a self-consistent Vlasov-Fokker-Planck equations which describes the longitudinal dynamics of an electron bunch in the storage ring of a synchrotron particle accelerator. We show existence and uniqueness of global classical solutions under physical hypotheses on the initial data. The proof relies on a mild formulation of the equation and hypoelliptic regularization estimates. We also address the problem of the long-time behavior  of solutions. We prove the existence of steady states,  called Haissinski solutions, given implicitly by a nonlinear integral equation. When the beam current (\emph{i.e.} the nonlinearity) is small enough, we show uniqueness of steady state and local asymptotic nonlinear stability of solutions in appropriate weighted Lebesgue spaces. The proof is based on hypocoercivity estimates. Finally, we discuss the physical derivation of the equation and its particular asymmetric interaction potential. 

\medskip

\noindent\textsc{Keywords: } Vlasov-Fokker-Planck; Hypocoercivity; Hypoellipticity; Synchrotron radiation; Wakefield; Haissinski solutions.

\medskip
\noindent\textsc{Mathematics Subject Classification (2020): 35Q83; 35Q84; 35B35; 82C22.}

\end{abstract}
\maketitle

\tableofcontents

\section{Introduction}
We are interested in the study of a Vlasov-Fokker-Planck (VFP) equation which models the dynamics of an electron bunch in the storage ring of a synchrotron particle accelerator. In these devices, the particles rotate at a velocity close to the speed of light in an almost circular motion. Thanks to various electromagnetic devices such as bending magnets, wigglers and radio-frequency (RF) cavity, the particles are kept close to the ideal trajectory of the so-called \emph{synchronous electron} and emit light in the form of synchrotron radiation. As each particle creates its own electromagnetic field the bunch is also subject self consistent interactions. 

In this context, the particles can be described by their distribution function $f\equiv f(t,x,v)$ evolving in time $t\geq0$. The phase space variables $(x,v)\in\RR^2$ are dimensionless quantities. The $x$ variable stands for the longitudinal position relative to the synchronous particle along the ideal orbit. The $v$ variable denotes the deviation from the nominal energy of the synchronous particle. The notation for phase space variables  usually differs in the physics literature (see for instance \cite{warnock2000general, cai2011linear, roussel2014spatio, warnock2018numerical}). Here for the sake of comparison with other models from mathematical kinetic theory we keep the $(x,v)$ notation even if it does account for absolute position and velocity of particles.

The Vlasov-Fokker-Planck system describing the dynamics reads
\begin{equation}\label{eq:VFPnonlin}
\left\{
\begin{array}{l}
\partial_t f + v\partial_xf +(F_f(t,x)-\alpha x)\partial_v f = 2\nu\partial_v(vf +\theta \partial_v f),\\[.75em]
\ds F_f(t,x) = -I\partial_x K\ast\rho(t,x) = -I\iint_{\RR^2} \partial_xK(x-y)f(t,y,w)\dd y\dd w,\\[.75em]
f(0,x,v) = f^\text{in}(x,v).
\end{array}
\right.
\end{equation}
The parameters $\alpha, \nu, \theta$ and $I$ are assumed to be positive and the latter will be referred to as beam current. The quantity 
$
\rho(t,x) = \int_{\RR} f(t,x,v)\dd v 
$ 
is the macroscopic density of particles. In the equation, the first transport term contributes to the offset of particles with higher or lower energy than that of the synchronous electron. The contribution $-\alpha x$ to the relative energy accounts for the linearized field of the RF cavity, which is tuned to confine the particles near the synchronous electron. The Fokker-Planck part of the equation on the right-hand side models the variation of energy due to synchrotron radiation. The damping term comes from a balance between the loss of energy due to synchrotron radiation and the gain due to the acceleration by the RF cavity. The diffusion part of the Fokker-Planck equation accounts for random quantum fluctuations in the energy, due to the emission of photons. Self-consistent electromagnetic interactions are encoded in the mean-field term $F_f(t,x)$ which is proportional to the longitudinal part of the electric field emitted by the particles, called the \emph{wakefield}. It depends on a potential $K\equiv K(x)$ which in practice is bounded with bounded but discontinous derivative meaning that typically $K\in W^{1,\infty}(\RR)$. Compared to, for instance, a Coulomb interaction potential for electrostatic plasmas, 
here $K$ \emph{does not need to be an even function} of the position. Typically there are $x\in\RR$ such that
\begin{equation}\label{eq:asym}
K(x)\neq K(-x).
\end{equation}
Equivalently, the Fourier transform of the \emph{wakefield potential} $-\partial_xK$, called the \emph{impedance} in this context, has a non trivial real part. The wakefield potential depends on the characteristics of the accelerator. However, to fix ideas, in the idealized case of an electron moving near the speed of light in free space on a circular orbit, $K$ is supported on the half line $\{x\geq0\}$ meaning that an electron interacts only with electrons ahead of it on the orbit. The asymmetry is due to the geometry of the trajectories and the relativistic nature of the dynamics.

This paper is devoted to the analysis of \eqref{eq:VFPnonlin} for a general asymmetric interaction kernel $K\in W^{1,\infty}$. We study the well-posedness of the model, regularity of solutions, steady states as well as the long-time behavior of solutions in a weakly nonlinear regime (\emph{i.e.} for $I$ small enough). Compared to similar models that are studied in the mathematical literature, because of \eqref{eq:asym}, equation~\eqref{eq:VFPnonlin} lacks the usual free energy estimate (see Section~\ref{sec:freeenerg}) which yields a Lyapunov functional. This is, from the mathematical point of view, the main originality of this model (see for instance \cite[Section 17]{villani_2009_hypocoercivity} in comparison). However, this seemingly harmless generalization of the model  changes dynamical properties of the equation in highly nonlinear regime, where complex behaviors such as instabilities or limit cycles can exist.  Let us also mention that the focus here is not put on the form of the external confinement  potential, which is taken quadratic here as it arises naturally in the derivation of the model. In Section~\ref{sec:derivation}, we give more details, as well as both formal and rigorous arguments concerning the derivation of equation~\eqref{eq:VFPnonlin} and the typical kernel $K$ in the context of storage rings.

Vlasov-Fokker-Planck equations have attracted a lot of attention since the seminal paper of Chandrasekhar \cite{chandrasekhar_1943_stochastic}. Different flavors are used  to model systems of interacting particles in astrophysics, plasma physics, particle accelerator physics... As in \eqref{eq:VFPnonlin} they combine a Vlasov part, \emph{i.e.} the left-hand side, and a Fokker-Planck part, \emph{i.e.} the right-hand side. The mean field interaction kernel $K$ can then encode a variety of self-consistent interactions. One of the most studied version is the Vlasov-Poisson-Fokker-Planck (VPFP) equation, which is \eqref{eq:VFPnonlin} -- and higher dimensional versions -- with $K$ the fundamental solution of either $-\Delta_x$ or $\Delta_x$, namely a Coulomb or a Newton interaction kernel. Because of the importance of VPFP in the literature, we will often compare \eqref{eq:VFPnonlin} with VPFP in the following.  Its well-posedness has been studied in \cite{neunzert_1984_VFP, degond_1986_global, victory_1990_classical, victory_1991_existence, bouchut_1993_existence}. In the present paper we will rely on the mild (Duhamel) formulation of \eqref{eq:VFPnonlin} in order to construct solutions, as in \cite{victory_1990_classical}. Indeed since the potential is quadratic one can use the explicit fundamental solution of the linear equation ($I=0$).

Vlasov-Fokker-Planck equations are degenerate parabolic equations since the diffusion occurs only in the $v$ variable. The question of the regularity of solutions via hypoellipticity theory \cite{hormander_1967_hypoelliptic} has been studied for linear VFP with various confining potentials \cite{herau_2004_isotropic, hellfer_2005_hypoelliptic}, as well as for VPFP \cite{bouchut_1995_smoothing}. For \eqref{eq:VFPnonlin}, we will once again use the fundamental solution to derive regularization estimates. Let us mention however, that methods purely based on energy estimates and functional analysis tools could also be used to derive regularization estimates (see for instance  \cite[Section A.21]{villani_2009_hypocoercivity}).

For the nonlinear Vlasov-Fokker-Planck \eqref{eq:VFPnonlin}, one expects a stationary solution of the form
\begin{equation}\label{eq:steadystate}
f^\infty(x,v) = \frac{\mathcal{M}}{\sqrt{2\pi\theta}}e^{-\frac{v^2}{2\theta}}\sigma^\infty(x)\,,
\end{equation}
with $\mathcal{M}$ the initial mass
\begin{equation}\label{eq:mass}
\mathcal{M} = \iint_{\RR^2} f^\text{in}(x,v)\dd x\dd v.
\end{equation}
In the context of VPFP, such states are called Maxwell-Boltzmann densities which can be found by solving a nonlinear Poisson equation to find $\sigma^\infty$ (see for instance \cite{dolbeault_1991_stationary} or \cite[Section 3]{herda_2018_large}). More generally when the potential is symmetric, there is a variational characterization of stationary solutions which can be studied with classical methods from calculus of variation. Under assumption \eqref{eq:asym}, this characterisation is not satisfied anymore and one determines $\sigma^\infty$ by finding a fixed point of a nonlinear and nonlocal mapping. In the particle physics literature, this fixed point equation (and their solutions) is called Haissinski equation (resp. solutions) in reference to \cite{haissinski1973exact}. Let us point out  \cite{warnock2018numerical} and references therein for further details on this topic.

Concerning the long-time behavior, in the linear case $I=0$ one easily shows convergence towards the unique stationary state (here a Gaussian) with rate (here exponential) depending on the confining potential (here quadratic). In more general linear settings, quantitative convergence estimates can be established thanks to hypocoercivity methods \cite{herau_2004_isotropic, villani_2009_hypocoercivity, DMS, bouin_2020_hypocoercivity}. Research in this direction has been very active in the past couple of decades. Hypocoercivity methods have also been used for nonlinear models. For \eqref{eq:VFPnonlin} with symmetric interactions, $K\in W^{2,\infty}$ and confinement on the torus (instead of through a quadratic potential), we refer to \cite[Section A.21]{villani_2009_hypocoercivity}. With the same Lipshitz assumption on $\partial_xK$ and more general confinements, we refer to the probabilistic approaches \cite{bolley_2010_trend, guillin_2022_convergence} relying on the SDE interpretation of \eqref{eq:VFPnonlin} through Langevin type equations. Concerning   VPFP, asymptotic stability in weakly nonlinear regimes and  perturbative setting is shown in \cite{herau_2016_global,herda_2018_large}. For specific kinetic models, hypocoercivity methods have been successful in nonlinear settings without smallness assumptions in  \cite{favre2020hypocoercivity} and in \cite{ADLT}. In the latter, which concerns VPFP in dimension $1$ (which is  \eqref{eq:VFPnonlin} with the symmetric kernel), the hypocoercivity method is coupled with an important a priori estimate which is the free energy estimate which yields a Lyapunov functional for the nonlinear equation. As soon as \eqref{eq:asym} is satisfied such a Lyapunov functional becomes unavailable and one expects asymptotic stability of the stationary solutions only in weakly nonlinear regime, namely when $I$ is small enough. This is what we are going to show thanks to a $L^2$ hypocoercivity method. In the strongly nonlinear regime, which is of high interest for applications, there are theoretical \cite{cai2011linear}, numerical \cite{venturini2005coherent} and experimental \cite{evain2017direct} evidences in the literature  of instabilities or time-periodic limit cycles. The former have been shown in some situations to be stable, or stabilizable via a control on the intensity of the confinement, namely the parameter $\alpha$ (see \cite{evain2019stable}). In the context of particle accelerator physics, these well-known complex behavior are referred to as microbunching instability.  The mathematical  study of this regime is out of the scope of the present paper.

Finally, let us mention that \eqref{eq:VFPnonlin} models only the longitudinal dependency of the electron bunch along the orbit. More complete models which include a transverse description have been developed \cite{cai2017coherent} in the context of particle accelerators. In different context there have also been many papers concerning the purely transverse description of beams \cite{degond_1993_paraxial,filbet2006modeling}.

\subsection*{Main results}

We first show the well-posedness and regularity of solutions of \eqref{eq:VFPnonlin}. 
\begin{theo}[Global well-posedness]\label{theo:well}
Let $f^\text{in}\in L^1(\RR^2)$ be non-negative and let $\partial_x K\in L^\infty(\RR)$. There exists a smooth classical solution  $f\in\mathcal{C}([0,\infty), L^1(\RR^2))\cap \mathcal{C}^{\infty}((0,\infty)\times\RR^2)$ to the Vlasov-Fokker-Planck equation \eqref{eq:VFPnonlin}. Moreover for $t>0$, the solution is non-negative, $f(t,x,v)\geq0$, and the total mass is conserved  
\[
\iint_{\RR^2} f(t,x,v) \dd x\dd v = \iint_{\RR^2} f^\text{in}(x,v) \dd x\dd v.
\]
Furthermore, for any $n,m\in\mathbb{N}$ there is $C_{T,n,m}>0$ such that for $t\in(0,T]$
\[
\|\partial_v^n\partial_x^mf(t)\|_{L^1(\RR^2)}\leq C_{T,n,m}\left(t^{-\frac n2-\frac{3m}{2}}+1\right)\|f^\text{in}\|_{L^1(\RR^2)}.
\]
Finally, the mapping $f^\text{in}\in L^1_+(\RR^2)\mapsto f\in\mathcal{C}([0,T], L^1_+(\RR^2))$ is Lipschitz continuous, where $L^1_+(\RR^2)$ denotes the cone of non-negative functions in $L^1(\RR^2)$.
\end{theo}

The proof of this result in Section~\ref{sec:proof_theo_1} is based on the construction of mild $\mathcal{C}([0,\infty), L^1(\RR^2))$ solutions using the explicit fundamental solution of the linear Vlasov-Fokker-Planck equation. Regularization estimates also rely on precise estimates on the fundamental solution which are then used in the Duhamel formulation of the nonlinear equation.

Then, we show the following theorem concerning stationary states of \eqref{eq:VFPnonlin}.
\begin{theo}[Steady states]\label{theo:steady}
For any $\mathcal{M}>0$ and $K\in W^{1,p}(\RR)$, $p\in[1,\infty]$, there exists a stationary solution $f^\infty$ of \eqref{eq:VFPnonlin} with mass $\mathcal{M}$ which is of the form \eqref{eq:steadystate} where $\sigma^\infty$ is a fixed point of the mapping
\[
\cT(\sigma)(x):=\dfrac{e^{-\frac{\alpha x^2}{2\theta}- \frac{I \mathcal{M}}{\theta}K\ast\sigma(x)}}{\int_\RR e^{-\frac{\alpha y^2}{2\theta} - \frac{I \mathcal{M}}{\theta}K\ast\sigma(y)}\dd y},
\]
which is such that $\sup_{x\in\RR}|\partial_x^n\sigma(x)e^{\frac{\beta x^2}{2}}|<\infty$ for any  $0<\beta<\alpha/\theta$ and $n\in\NN$. Moreover, for any $K\in L^\infty(\RR)$, there is $C>0$ such that if 
\[
I < I^\text{thres}:=\frac{C \theta}{\mathcal{M}\|K\|_{L^\infty}},
\]
then there is a unique integrable stationary solution of the form \eqref{eq:steadystate}.
\end{theo}

The proof of this result in Section~\ref{sec:proof_theo_2} relies on the use of a Schaefer fixed point theorem. The key element is to define the right fixed point formulation, \emph{i.e.} the right mapping $\cT$. This normalized formulation is inspired by the work of Warnock et. al. \cite{warnock2000general, warnock2018numerical}. It allows for a free $L^1$ bound on $\cT(\sigma)$, which is crucial in the proof. Let us mention that a similar uniqueness result was claimed in \cite{warnock2000general}.

{
Finally, we prove the local nonlinear asymptotic stability of Haissinski steady states. Given $f^\infty$, we introduce the $L^2$ space for the measure $\dd\mu^\infty = (f^\infty)^{-1}\dd x\dd v$, characterized by its norm 
\[
\|g\|_{L^2(\mu^\infty)} = \|g(f^\infty)^{-\frac12}\|_{L^2}.
\]
Then we have the following result.
\begin{theo}[Long-time behavior]\label{theo:long}
Let $f^\infty$ be a steady state given by Theorem~\ref{theo:steady} and $f$ be the solution of \eqref{eq:VFPnonlin}. There are constants $C_1, C_2$ that depend only on $\alpha,\theta,\nu$ and $\|K\|_{W^{1,\infty}}$ such that if the initial data is such that
\begin{equation}\label{assumption}
I \mathcal{M} \leq C_1,
\end{equation}
and
\begin{align} \label{eq:Assumptiongin}
    \sqrt{\mathcal{M}}\|f^\text{in}-f^\infty \|_{L^2(\mu^\infty)} \leq C_2,
\end{align}
then there exists $\lambda$ that depend only on $\alpha,\theta,\nu$ and $\|K\|_{W^{1,\infty}}$, as well as a constant $C_3>1$ depending additionally on $\mathcal{M}$ such that
\begin{align*}
    ||f(t)-f^\infty\|_{L^2(\mu^\infty)} \leq C_3 e^{-\lambda t}\|f^\text{in}-f^\infty \|_{L^2(\mu^\infty)}\,,\quad \text{for}\ t\geq0.
\end{align*}
\end{theo}}

The proof of this result, in Section~\ref{sec:proof_theo_3} is based on the hypocoercivity framework of \cite{DMS}. The main technical challenge lies in the control of terms (both linear and nonlinear) coming from mean field interactions. In particular in order to derive the weak nonlinearity condition \eqref{assumption} and close estimates, one needs to precisely track the dependence of various constants on the parameters of the model (see for instance Lemma~\ref{lem:LinHypo}). All the constants appearing in Theorem~\ref{theo:long}, and their dependance on the parameters of the model, can be found in our proofs, in particular the dependence of the decay rate. However, let us mention that it is known that hypocoercivity estimates, even in simple cases do not yield sharp decay rates, even after carefully optimizing each inequality (see \cite{DMS}).

The locality condition \eqref{eq:Assumptiongin} can be replaced by a more restrictive condition on the smallness of mass $\mathcal{M}\leq C_2/2$, by using the embedding $L^2(\mu^\infty)\subset L^1(\RR^2)$. 

 Observe that we prove exponential decay to equilibrium in a smaller space than that of our well-posedness result in Theorem~\ref{theo:well}, since functions in $L^2(\mu^\infty)\subset L^1(\RR^2)$ decay at least like Gaussians at infinity.  Thanks to the regularization properties of \eqref{eq:VFPnonlin} and the space enlargement techniques developed in \cite{mishler_2016_exponential, gualdani2017factorization} it might be possible to extend our result to slowly decaying initial data.

Finally, the last results of this paper are contained in Section~\ref{sec:derivation} which is dedicated to the derivation of \eqref{eq:VFPnonlin}, and more precisely the derivation of the interaction kernel from Maxwell equations. We formalize mathematically  in Proposition~\ref{prop:mur} and Proposition~\ref{prop:approx}, the derivation of the free space wakefield (due to Murphy, Krinsky and Gluckstern \cite{murphy1996longitudinal}) which is the typical example of asymmetric interaction kernel $K$ arising in particle accelerator physics (see Figure~\ref{fig:wakefield} in Section~\ref{sec:derivation} for an illustration).

\subsection*{Outline}

The rest of the paper is organized as follows. In Section~\ref{sec:cauchy} we prove the global well-posedness result of Theorem~\ref{theo:well}. In Section~\ref{sec:long}, we are concerned with the long-time behavior of solutions, and prove Theorem~\ref{theo:steady} and Theorem~\ref{theo:long}. Finally, in Section~\ref{sec:derivation}, we give details about the derivation of the model and the wakefield potential in particular in Proposition~\ref{prop:mur} and Proposition~\ref{prop:approx}.

\subsubsection*{Acknowledgements} The authors warmly thank Serge Bielawski, Clément Evain, Eléonore Roussel and Christophe Szwaj for enlightening discussions on the model \eqref{eq:VFPnonlin} in the context of particle accelerator physics.\\
M.H. acknowledges support from the CYNA project of the CY Initiative of Excellence. 

\section{The Cauchy problem}\label{sec:cauchy}
In this section we address the well-posedness of the nonlinear Vlasov-Fokker-Planck equation \eqref{eq:VFPnonlin}. We build mild solutions to the equation using the explicit formula for the fundamental solution of the Vlasov-Fokker-Planck equation. Then we show that any mild solution is actually a smooth classical solution. 
\subsection{Fundamental solution of linear VFP}

We first compute the fundamental solution of the Vlasov-Fokker-Planck equation with  a confining potential. This type of computations are classical and date back to Chandrasekhar \cite{chandrasekhar_1943_stochastic}. The formula for the fundamental solution of VFP without external forcing term can be found for instance in \cite{chandrasekhar_1943_stochastic, victory_1991_existence, bouchut_1993_existence}. Here we compute the $G\equiv G(\tau,x,y,v,w)$ solving  
\begin{equation}\label{eq:VFPfundamental}
\partial_\tau G + v\partial_xG -\alpha x\partial_v G = 2\nu\partial_v(vG +\theta \partial_v G)\,, \quad \tau>0,\ (x,v)\in\RR^2
\end{equation}
and such that for any integrable function $f^\text{in}\in L^1(\RR^2)$ one has 
\[
\lim_{\tau\to0^+}\iint_{\RR^2}G(\tau,x,y,v,w)f^\text{in}(y,w)\dd y \dd w = f^\text{in}(x,v),\quad \text{a.e. }(x,v)\in\RR^2
\]
In other words $G(0,x,y,v,w) = \delta_0(x-y)\delta_0(v-w)$.

\begin{prop}\label{prop:fundamental}
The fundamental solution of \eqref{eq:VFPfundamental} is given by the following bivariate Gaussian
\begin{equation}\label{eq:Gformula}
G = \frac{1}{2\pi\sqrt{\mathrm{det}\Sigma(\tau)}}\exp\left(-\frac{1}{2}\mu^\top \Sigma(\tau)^{-1}\mu\right),\quad
\text{where} \quad \mu = \binom{x}{v} - B(\tau)\binom{y}{w},
\end{equation}
and the matrices $B(\tau)$ and $\Sigma(\tau)$ are defined as follows. Let $\omega$ real or pure imaginary such that $\omega^2 = \alpha-\nu^2$ (take any root) and define $s(\tau) = \frac{\sin(\omega\tau)}{\omega}$ (if $\omega = 0$ then $s(\tau)=\tau$) and $c(\tau) = \cos(\omega \tau)$. Then
\[
\qquad B(\tau) = e^{-\nu\tau}\left(\begin{matrix}
 c( \tau) + \nu s( \tau)
&s( \tau)\\
- \alpha s( \tau) &  c( \tau) - \nu s( \tau)
\end{matrix}\right)\,,
\]
and the covariance matrix is 
\[
\Sigma(\tau) = \left(\begin{matrix}\theta/\alpha&0\\
0&\theta
\end{matrix}\right)-\theta e^{-2\nu\tau}\left(\begin{matrix}s^2( \tau) + (\nu s( \tau)+ c( \tau))^2/\alpha&-2\nu s^2( \tau)\\
-2\nu s^2( \tau)&\alpha s^2( \tau)+(\nu s( \tau)- c( \tau))^2
\end{matrix}\right)\,,
\]
with determinant
\[
\det(\Sigma(\tau)) = \frac{4\theta^2}{\alpha} e^{-2\nu\tau}\left(\sinh^2(\nu\tau)-\nu^2s^2(\tau)\right).
\]
\end{prop}
\begin{proof}
In Fourier variables $\hat{G}\equiv \hat{G}(\tau,\eta,y,\xi,w)$ solves
\[
%\left\{
\begin{array}{l}
\partial_\tau \hat{G} - \eta\partial_\xi \hat{G} +\alpha \xi\partial_\eta \hat{G} = -2\nu\xi\partial_\xi\hat{G} -2\nu\theta \xi^2 \hat{G}\,,
\end{array}
%\right.
\]
starting at $\hat{G}(0,\eta,y,\xi,w) = \exp(-i\eta y-i\xi w)$. Define
\[A = \left(\begin{matrix}
0&1\\
-\alpha&-2\nu
\end{matrix}\right),\quad\text{and}\quad B(\tau) = e^{A\tau}.
\]
Using the method of characteristics, one finds that this transport equation is readily solved by
\[
\hat{G}(\tau,\eta,y,\xi,w) = \exp\left(-iyN(\tau,\eta,\xi)-iwZ(\tau,\eta,\xi)-2\nu\theta\int_0^\tau Z(s,\eta,\xi)^2\dd s \right)
\]
with 
\[
\binom{N}{Z}(\tau,\eta,\xi) = B(\tau)^\top\binom{\eta}{\xi}.
\]
Let $B_2(\tau)$ be the second column of the matrix $B(\tau) = (B_1(\tau), B_2(\tau))$ and define 
\[\Sigma(\tau) = 4\nu\theta\int_0^\tau B_2(\tau)B_2(\tau)^\top\dd s\]
Then $\hat{G}$ is given by the Gaussian
\[
\hat{G}(\tau,\eta,y,\xi,w) = \exp\left(-i\binom{\eta}{\xi}^\top B(\tau)\binom{y}{w} -  \frac12\binom{\eta}{\xi}^\top\Sigma(\tau) \binom{\eta}{\xi}\right),
\]
which yields the result, by inverse Fourier transform.

\end{proof}

An immediate consequence of  \eqref{eq:Gformula} is the following.

\begin{rema}[Long time behavior]
Observe that for all $x,y,v,w\in\RR$ one has
\[
\lim_{\tau\to\infty}G(\tau,x,y,v,w) = \frac{\sqrt{\alpha}}{2\pi\theta}\exp\left(-\frac{\alpha x^2}{2\theta}-\frac{v^2}{2\theta}\right).
\]
\end{rema}

From the explicit formula \eqref{eq:Gformula} we now derive estimates on the $G$.

\begin{prop}[Kernel estimates]\label{prop:ker}
For all $\tau>0$ and $x,y,v,w,\in\RR$, \[G(\tau,x,y,v,w)>0\] and 
\begin{equation}\label{eq:massconservGdwdy}
\iint_{\RR^2}G(\tau,x,y,v,w)\dd x\dd v = 1,
\end{equation}
\begin{equation}\label{eq:LinfgrowthGdwdy}
\iint_{\RR^2}G(\tau,x,y,v,w)\dd y\dd w = e^{2\nu\tau}.
\end{equation}
Moreover, for any $p\in[1,\infty]$
\begin{equation}\label{eq:Lpkernelxv}
\|G(\tau)\|_{L^\infty_{y,w}(L^p_{x,v})} \lesssim_{\theta,\alpha,\nu}\tau^{-2(1-\frac1p)}+1,
\end{equation}
and
\begin{equation}\label{eq:Lpkernelyw}
\|G(\tau)\|_{L^\infty_{x,v}(L^p_{y,w})} \lesssim_{\theta,\alpha,\nu}\left(\tau^{-2(1-\frac1p)}+1\right)e^{\frac{2\nu\tau}{p}},
\end{equation}
\end{prop}

\begin{proof}
The integrals are straightforward (just notice that $\det(B(\tau)) = e^{-2\nu\tau}$). Then we need a lower bound on $\det(\Sigma(\tau))$.  Observe that for all $\tau>0$, $\det(\Sigma(\tau))>0$. Moreover $\det(\Sigma(\tau)) = O(\tau^4)$ when $\tau\to 0$ and $\lim_{\tau\to\infty}\det(\Sigma(\tau)) = \theta^2/\alpha$. Thus $\det(\Sigma(\tau))\gtrsim \min(\tau^4,1)$. Therefore, for all $x,y,v,w,\in\RR$ one has
\[
|G(\tau,x, y, v, w)| \lesssim_{\theta,\alpha,\nu}\left(\frac{1}{\tau^2}+1\right).
\]
which yields \eqref{eq:Lpkernelxv} and \eqref{eq:Lpkernelyw} by interpolation between Lebesgue spaces.
\end{proof}

\begin{prop}[Kernel partials estimates]\label{prop:kerder}
For all $\tau>0$ and $v,w\in\RR$ one has
\begin{equation}\label{eq:deriv_massconservGdwdy}
\iint_{\RR^2}\partial_wG(\tau,x,y,v,w)\dd x\dd v = \iint_{\RR^2}\partial_yG(\tau,x,y,v,w)\dd x\dd v = 0
\end{equation}
Moreover for all $a,b,c,d\in\mathbb{N}$
\begin{equation}\label{eq:estimGderivatives}
\|\partial_x^a\partial_v^b\partial_y^c\partial_w^dG\|_{L^\infty_{x,v}(L^{1}_{y,w})}\lesssim \left(\tau^{-\frac{3a}{2}-\frac{b}{2}}+1\right)\tau^{-\frac{3c}{2}-\frac{d}{2}}
\end{equation}
\end{prop}
\begin{proof}
Let $e_1 = \binom{1}{0}$ and $e_2 = \binom{0}{1}$ be the canonical basis of $\RR^2$. Using the explicit formula \eqref{eq:Gformula} for the kernel, one computes

\[
\partial_wG = (-\partial_w\mu^\top\Sigma^{-1}\mu) G = (-e_2^\top B^\top\Sigma^{-1}\mu) G,
\]
which vanishes after integration in $x,y$ (since it is proportional to the expectation of a centered Gaussian). The same happens for $\partial_yG$ and thus it proves \eqref{eq:deriv_massconservGdwdy}.
To compute higher order derivatives let us write, for combinatorial purposes, $\partial_1\mu = \partial_x\mu$, $\partial_2\mu = \partial_v\mu$, $\partial_3\mu = \partial_y\mu$, $\partial_4\mu = \partial_w\mu$ and let us define $a_{ij} = -\partial_i\mu^\top\Sigma^{-1}\partial_j\mu$ and $b_k = -\partial_k\mu^\top\Sigma^{-1}\mu$. Then observe that the ratio $\partial_x^a\partial_v^b\partial_y^c\partial_w^d G / G$ is given by the sum of all the products between $a_{ij}$'s and $b_k$'s with $i,j,k$ ranging from $1$ to $4$ and such that the indices $1,2,3,4$ appear exactly $a,b,c,d$ times respectively in the product. In particular, there is a polynomial $p$ of degree $a+b+c+d$ such that
\[
|\partial_x^a\partial_v^b\partial_y^c\partial_w^d G|\leq 
|\Sigma^{-\frac12}\partial_x\mu|^a|\Sigma^{-\frac12}\partial_v\mu|^b|\Sigma^{-\frac12}\partial_y\mu|^c|\Sigma^{-\frac12}\partial_w\mu|^d p(|\Sigma^{-\frac12}\mu|) G.
\]
 From there asymptotic expansions reveal that
 \[
|\Sigma^{-\frac12}\partial_x\mu|^2 = e_1^\top\Sigma(\tau)^{-1}e_1 
= \left\{
 \begin{array}{ll}
 \ds\frac{3}{\nu\theta\tau^3} + o(\tau^{-3})&\text{as }\tau\to0,\\[.75em]
 \ds\frac{\alpha}{\theta} + o(1)&\text{as }\tau\to \infty.
 \end{array}\right.
 \]
 \[
 |\Sigma^{-\frac12}\partial_v\mu|^2 = e_2^\top \Sigma(\tau)^{-1}e_2
  = \left\{
 \begin{array}{ll}
 \ds\frac{1}{\nu\theta\tau} + o(\tau^{-1})&\text{as }\tau\to0,\\[.75em]
 \ds \frac{1}{\theta} + o(1)&\text{as }\tau\to \infty.
 \end{array}\right.
\]
 \[
 |\Sigma^{-\frac12}\partial_y\mu|^2 = e_1^\top B(\tau)^\top\Sigma(\tau)^{-1}B(\tau)e_1 = \left\{
 \begin{array}{ll}
 \ds\frac{3}{\nu\theta\tau^3} + o(\tau^{-3})&\text{as }\tau\to0,\\[.75em]
 \ds o(\tau^{-3})&\text{as }\tau\to \infty.
 \end{array}\right.
\]
 \[
 |\Sigma^{-\frac12}\partial_w\mu|^2 = e_2^\top B(\tau)^\top\Sigma(\tau)^{-1}B(\tau)e_2 = \left\{
 \begin{array}{ll}
 \ds\frac{1}{\nu\theta\tau} + o(\tau^{-1})&\text{as }\tau\to0,\\[.75em]
 \ds o(\tau^{-1})&\text{as }\tau\to \infty.
 \end{array}\right.
 \]
 This proves \eqref{eq:estimGderivatives} with a constant depending only on $a,b,c,d,\alpha,\nu,\theta$.
\end{proof}

\begin{rema}
It is well-known that for the VFP equation \eqref{eq:VFPfundamental} is hypoelliptic \cite{hormander_1967_hypoelliptic, herau_2004_isotropic, hellfer_2005_hypoelliptic}. In the regularization estimate \eqref{eq:estimGderivatives}, the hypoelliptic nature of the PDE can be seen through the singularity at $\tau\to0^+$. Each $\partial_v$ and $\partial_w$ derivatives bring a $O(\tau^{-1/2})$ singularity in the estimate so the regularization properties are similar to that of the heat equation, because of the diffusion in the $v$ variable. Conversely each $\partial_x$ and $\partial_y$ derivatives bring a bigger $O(\tau^{-3/2})$ singularity, reflecting the degeneracy of the kinetic Fokker-Planck operator and the interplay between transport in the $x$ variable and diffusion in the $v$ variable.
\end{rema}

\subsection{Mild solution of the VFP equation}

Let us first consider the linear VFP equation
\begin{equation}\label{eq:VFPlin}
\left\{
\begin{array}{l}
\partial_t f + v\partial_xf +(F(t,x)-\alpha x)\partial_v f = 2\nu\partial_v(vf +\theta \partial_v f),\\[.75em]
f(0,x,v) = f^\text{in}(x,v).
\end{array}
\right.
\end{equation}
with a given force field $F$. We build a mild solution for this equation using the explicit representation of the Vlasov-Fokker-Planck semigroup through its fundamental solution. The method is  strongly inspired by the papers of Victory, O'Dwyer  \cite{victory_1990_classical} and Bouchut \cite{bouchut_1993_existence}.
\begin{defi}
We say that $f\in\mathcal{C}([0,\infty), L^1(\RR^2))$ is a global mild solution of \eqref{eq:VFPlin} with force field $F\in L^\infty([0,\infty)\times\RR)$ and initial data $f^\text{in}\in L^1(\RR^2)$ if for all $t\geq0$ and a.e. $(x,v)\in\RR^2$ one has 
\begin{multline*}
f(t,x,v) = \iint_{\RR^2}G(t,x,y,v,w)f^\text{in}(y,w)\dd y\dd w \\
+ \int_0^t\iint_{\RR^2}\partial_wG(t-s,x,y,v,w)F(s,y)f(s,y,w)\dd y\dd w\dd s.
\end{multline*}
\end{defi}

\begin{prop}\label{prop:mildsol_linVFP}
Given $F\in L^\infty([0,\infty)\times\RR)$ and $f^\text{in}\in L^1(\RR^2)$, there is a unique mild solution $f\in\mathcal{C}([0,\infty), L^1(\RR^2))$ of the linear Vlasov-Fokker-Planck equation \eqref{eq:VFPlin}. This solution conserves mass
\[
\iint_{\RR^2}f(t,x,v)\dd x\dd v = \iint_{\RR^2}f^\text{in}(x,v)\dd x\dd v,\quad \forall t\geq0.
\]
and satisfies  the additional estimate
\begin{equation}\label{eq:estimdvf_VFPlin}
\|\partial_v f(t)\|_{L^1_{x,v}}\lesssim (1+t^{-\frac12}).
\end{equation}
\end{prop}
\begin{proof}
Let $T>0$ and let us define the operators $\mathcal{G}$ and $\mathcal{H}_F$ acting on \[\mathcal{Y}_T = \mathcal{C}([0,T], L^1(\RR^2))\] by
\[
\mathcal{G}f(t,x,v) = \iint_{\RR^2}G(t,x,y,v,w)f(t,y,w)\dd y\dd w,
\]
\[
\mathcal{H}_Ff(t,x,v) = \int_0^t\iint_{\RR^2}\partial_wG(t-s,x,y,v,w)F(s,y)f(s,y,w)\dd y\dd w\dd s.
\]
Then using Proposition~\ref{prop:ker} with $p=1$ and \eqref{eq:estimGderivatives} one obtains
\[
\|\mathcal{G}f\|_{\mathcal{Y}_T}\leq\|f\|_{\mathcal{Y}_T},\quad
\|\mathcal{H}f\|_{\mathcal{Y}_T}\lesssim T^{\frac12}\|F\|_{L^\infty_{t,x}}\|f\|_{\mathcal{Y}_T}
\]
so that $\mathcal{G}$ and $\mathcal{H}_F$ are well defined and bounded. Moreover, using \eqref{eq:estimGderivatives} repeatedly, one has for some constant $C>0$ depending only on $\alpha, \theta$ and $\nu$ that, for all $s_0\geq0$,
\[
\|\mathcal{H}_F^nf\|_{L^1_{x,v}}(s_0)\leq (C\|F\|_{L^\infty_{t,x}})^n\alpha_n(s_0)\|g\|_{\mathcal{Y}_T}
\]
where
\[
\alpha_n(s_0) = \int_0^{s_0}\cdots\int_0^{s_{n-1}}\prod_{i=1}^n\frac{\dd s_i}{\sqrt{s_{i-1}-s_i}} = s_0^{\frac n2}\prod_{i=1}^n\int_0^1\frac{u^{\frac{i-1}{2}}}{\sqrt{1-u}}\dd u = \frac{(\pi s_0)^\frac n2}{\Gamma\left(\frac n2 + 1\right)}.
\]
Since the right-hand side of the last inequality defines a normally convergent series for any $s_0>0$ and since $\mathcal{Y}_T$ is a Banach space one can define the sum
\[
f := \sum_{n = 0}^\infty\mathcal{H}_F^n\mathcal{G}f^\text{in}\,,
\]
which by definition solves $f = \mathcal{G}f^\text{in} + \mathcal{H}_F f$ and is therefore a mild solution of \eqref{eq:VFPlin}. Uniqueness is readily obtained by linearity of the equation and the estimate on $\mathcal{H}_F$ and conservation of mass follows from \eqref{eq:massconservGdwdy} and \eqref{eq:deriv_massconservGdwdy}. For the estimate on the velocity derivative observe that it solves
\begin{multline*}
\partial_vf(t,x,v) = \iint_{\RR^2}\partial_vG(t,x,y,v,w)f^\text{in}(y,w)\dd y\dd w \\
- \int_0^t\iint_{\RR^2}\partial_vG(t-s,x,y,v,w)F(s,y)\partial_wf(s,y,w)\dd y\dd w\dd s.
\end{multline*}
Therefore as for the construction of $f$, by iterating the right-hand side one finds the estimate, which essentially relies upon \eqref{eq:estimGderivatives} and the integrability of the time singularity.
\end{proof}

Next we investigate uniform almost everywhere bounds on the solution. The non-negativity of density is easily proved formally using the PDE, but is not obvious on the mild formulation. In order to get back to the differential formulation we use a mollification argument inspired by Lions and Masmoudi \cite{lions_2001_uniqueness}.

\begin{lem}[$L^\infty$ bounds and non-negativity]
Let $f$ be the mild solution of \eqref{eq:VFPlin}. On the one hand, if $f^\text{in}\geq0$ then $f(t)\geq0$ for all $t\geq0$. As a consequence $\|f(t)\|_{L^1_{x,v}} = \|f^\text{in}\|_{L^1_{x,v}}$ for all $t\geq0$. On the other hand if $f^\text{in}\in L^\infty(\RR^2)$ then $f\in L^\infty_{\text{loc}}([0,\infty), L^\infty_{x,v})$.
\end{lem}
\begin{proof}
Let us start with the second property. If $f^\text{in}\in L^\infty$ then by \eqref{eq:LinfgrowthGdwdy} and \eqref{eq:estimdvf_VFPlin}
\[
\|f(t)\|_{L^\infty}\lesssim e^{2\nu t}\|f^\text{in}\|_{L^\infty} + \|F\|_{L^\infty}\int_{0}^t(1+s^{-\frac12})e^{2\nu s}\dd s.
\]
which proves the claim. For the non-negativity, we assume without loss of generality that $f^\text{in}\in L^1\cap L^\infty$. One recovers the $L^1$ case \emph{a posteriori} by a density argument. Assume that $f^\text{in}\geq0$  and let $f_\varepsilon$ be a mollification of $f$ such that $\|f_\varepsilon - f\|_{L^\infty(0,T,L^\infty_{x,v})}\to0$ when $\varepsilon\to0$, $\|\partial_vf_\varepsilon(t)\|_{L^1}\leq\|\partial_vf(t)\|_{L^1}$ and $f^\text{in}_\varepsilon\geq0$. Let us also introduce $g_\varepsilon = \mathcal{G}f^\text{in}_\varepsilon + \mathcal{H}_Ff_\varepsilon$. Observe that by the boundedness properties of the operators $\mathcal{G}$ and $\mathcal{H}_F$ one has $\|g_\varepsilon-f_\varepsilon\|_{L^\infty(0,T,L^\infty_{x,v})}\to 0$. Moreover $g_\varepsilon$ satisfies 
\[
\partial_tg_\varepsilon + v\partial_xg_\varepsilon  -\alpha x\partial_v g_\varepsilon  = 2\nu\partial_v(vg_\varepsilon  +\theta \partial_v g_\varepsilon) - F\partial_vf_\varepsilon.
\]
Multiplying the equation by $-2g_\varepsilon^- = -2\max(-g_\varepsilon,0)$ and integrating in all variables it is rather straightforward to obtain the estimate
\begin{multline*}
\|g_\varepsilon^{-}(t)\|_{L^2_{x,v}}^2\leq\|g_\varepsilon^{-}(0)\|_{L^2_{x,v}}^2 +2\nu\int_0^t\|g_\varepsilon^{-}(s)\|_{L^2_{x,v}}^2\dd s \\+ 2\|F\|_{L^\infty_{t,x}}\|g_\varepsilon-f_\varepsilon\|_{L^\infty_{t,x,v}}\int_0^t\|\partial_vf_\varepsilon(s)\|_{L^1}\dd s, 
\end{multline*}
where we used that $2(\partial_vf_\varepsilon) g_\varepsilon^- = 2(\partial_vf_\varepsilon) (g_\varepsilon^- - f_\varepsilon^-) - 2\partial_v((f_\varepsilon^-)^2)$.  The first term of the right-hand side is null since $g_\varepsilon(0) = f^\text{in}_\varepsilon\geq0$ and uniformly on finite time intervals the last term tends to $0$ as $\varepsilon$ tends to $0$. Therefore using a Grönwall argument and taking the limit $\varepsilon\to0$ eventually proves that $f\geq0$. 
\end{proof}

From the resolution of the VFP equation with a given force field we now build a solution to the nonlinear equation by an iteration argument.

\begin{prop}\label{prop:mildsol_nonlinVFP}
Assume that the interaction potential $K\equiv K(x)$ is such that $\partial_xK\in L^\infty(\RR)$ and let $f^\text{in}$ be a non-negative integrable function. Then there is a unique mild solution $f\in\mathcal{C}([0,\infty), L^1(\RR^2))$ of the nonlinear Vlasov-Fokker-Planck equation
 \eqref{eq:VFPnonlin}. It conserves mass and non-negativity.  Moreover the mapping $f^\text{in}\in L^1_+(\RR^2)\mapsto f\in\mathcal{C}([0,T], L^1_+(\RR^2))$ is Lipschitz continuous, where $L^1_+(\RR^2)$ denotes the cone of non-negative functions in $L^1(\RR^2)$.
\end{prop}
\begin{proof}
We use an iteration scheme. Let $f^0 = f^\text{in}$ and define $f^{n+1}$ as the mild solution of \eqref{eq:VFPlin} with the force field $F^n = -I\partial_x K\ast\rho^n$. Let us show that $(f^n)_n$ is a Cauchy sequence in $\mathcal{Y}_T = \mathcal{C}([0,T], L^1(\RR^2))$. Since 
\[
f^{n+1} - f^n = \mathcal{H}_{F^n}(f^{n+1} - f^n) +  (\mathcal{H}_{F^{n}}-\mathcal{H}_{F^{n-1}})f^{n}
\]
One has 
\begin{multline*}
\|f^{n+1}(t) - f^n(t)\|_{L^1_{x,v}}\leq\int_0^t\frac{C}{\sqrt{t-s}}\Big(\|f^{n+1}(s) - f^n(s)\|_{L^1_{x,v}}\|F^n(s)\|_{L^\infty_{x}}\\ +\|F^{n}(s) - F^{n-1}(s)\|_{L^\infty_{x}}\|f^{n}(s)\|_{L^1_{x,v}} \Big)\dd s
\end{multline*}
which by Grönwall yields 
\begin{multline*}
\|f^{n+1}(t) - f^n(t)\|_{L^1_{x,v}}\\\leq \exp\left(\int_0^t\frac{C}{\sqrt{t-s}}\|F^n(s)\|_{L^\infty_{x}}\dd s\right)\int_0^t\frac{C}{\sqrt{t-s}} \|F^{n}(s) - F^{n-1}(s)\|_{L^\infty_{x}}\|f^{n}(s)\|_{L^1_{x,v}}\dd s .
\end{multline*}
Using non-negativity and conservation of mass one has $\|f^{n}(s)\|_{L^1_{x,v}}=\|f^{\text{in}}\|_{L^1_{x,v}}$. Combining this with Young's convolution inequality one has $\|F^n(s)\|_{L^\infty_{x}}\leq I \|\partial_x K\|_{L^\infty}\|f^{\text{in}}\|_{L^1_{x,v}}$. Therefore
\[
\|f^{n+1}(t) - f^n(t)\|_{L^1_{x,v}}\leq C_T\int_0^t\frac{1}{\sqrt{t-s}} \|f^{n}(s) - f^{n-1}(s)\|_{L^1_{x,v}}\dd s
\]
with $C_T= CI\|\partial_x K\|_{L^\infty}\|f^{\text{in}}\|_{L^1_{x,v}}\exp({C\sqrt{T}I \|\partial_x K\|_{L^\infty}\|f^{\text{in}}\|_{L^1_{x,v}} })$. Iterating this bound one obtains,
\[
\|f^{n+1}(t) - f^n(t)\|_{L^1_{x,v}}\leq \frac{2 C_T^n (\pi T)^{\frac n2}\|f^{\text{in}}\|_{L^1_{x,v}}}{\Gamma(\frac n2 + 1)}\,.
\]
Since the right-hand side is summable $(f^n)_n$ is a Cauchy sequence in the Banach space $\mathcal{Y}_T$ and its limit $f$ is a non-negative mild solution of \eqref{eq:VFPnonlin}. Lipschitz continuity (and uniqueness) w.r.t. the initial data can be obtained with a Grönwall type argument on the same type of estimates.
\end{proof}

\subsection{Hypoelliptic regularity estimates}

In this section we investigate the regularization properties of \eqref{eq:VFPnonlin}.

\begin{prop}\label{prop:regu} Under the assumptions of Proposition~\ref{prop:mildsol_nonlinVFP}, for all $T>0$ and $n,m\in\mathbb{N}$ there is $C_{T,n,m}>0$ such that for $t\in(0,T]$
\[
\|\partial_v^n\partial_x^mf(t)\|_{L^1_{x,v}}\leq C_{T,n,m}\left(t^{-\frac n2-\frac{3m}{2}}+1\right)\|f^\text{in}\|_{L^1_{x,v}}.
\]
\end{prop}
\begin{proof}
The result can be proved by induction on the derivatives. We do not fully detail the proof but just the main ideas and crucial points. As a preliminary comment observe that $f\mapsto F_f$ is a bounded operator from $L^1_v(W^{s,1}_x)$ to $W^{s,\infty}_x$ for any $s\in\mathbb{N}$ with  norm equal to $I \|\partial_xK\|_{L^\infty}$. Then, similarly to the proof of Proposition~\ref{prop:mildsol_linVFP} let us notice that $\partial_vf$ should solve 
\begin{multline*}
\partial_vf(t,x,v) = \iint_{\RR^2}\partial_vG(t,x,y,v,w)f^\text{in}(y,w)\dd y\dd w \\
- \int_0^t\iint_{\RR^2}\partial_vG(t-s,x,y,v,w)F_f(s,y)\partial_wf(s,y,w)\dd y\dd w\dd s
\end{multline*}
but using the estimate of Proposition~\ref{prop:kerder} and the same type of iteration argument as Proposition~\ref{prop:mildsol_linVFP}. one builds a unique solution to this fixed point equation which satisfies the desired estimate. 

For the $x$ derivative observe from formula \eqref{eq:Gformula} that 
if one defines 
\[
B(\tau)^{-1} = \left(
\begin{matrix}
b_{yx}(\tau)&b_{yv}(\tau)\\
b_{wx}(\tau)&b_{wv}(\tau)\\
\end{matrix}
\right)\sim_{\tau\to0}
\left(
\begin{matrix}
1&-\tau\\
\alpha\tau&1\\
\end{matrix}
\right)
\]
then $\partial_xG(\tau) = -b_{yx}(\tau)\partial_yG(\tau) - b_{wx}(\tau)\partial_wG(\tau)$, therefore $\partial_xf$ must be the solution of the linear fixed point problem
\begin{multline*}
\partial_xf(t,x,v) = \iint_{\RR^2}\partial_xG(t,x,y,v,w)f^\text{in}(y,w)\dd y\dd w \\
+ \int_0^t\iint_{\RR^2}b_{yx}(t-s)\partial_wG(t-s,x,y,v,w)F_{\partial_yf}(s,y)f(s,y,w)\dd y\dd w\dd s\\
+ \int_0^t\iint_{\RR^2}b_{wx}(t-s)\partial_wG(t-s,x,y,v,w)F_f(s,y)\partial_wf(s,y,w)\dd y\dd w\dd s\\
+ \int_0^t\iint_{\RR^2}b_{yx}(t-s)\partial_wG(t-s,x,y,v,w)F_f(s,y)\partial_yf(s,y,w)\dd y\dd w\dd s,
\end{multline*}
which again can be solved by the estimate of Proposition~\ref{prop:kerder} and the same type of iteration argument as Proposition~\ref{prop:mildsol_linVFP}. Observe that the key element is the singularity of each term of the right hand side (except the first one) at $s\to t^-$ and $s\to0^+$, which must be integrable to close the estimate.

The argument can be pursued for higher order derivatives with a small modification to take into account the growing singularity appearing in the lower order derivatives. For instance one seeks for $\partial_{xv}^2f$ as a solution of 
\begin{multline*}
\partial_{xv}^2f(t,x,v) = \iint_{\RR^2}\partial_{xv}^2G(t,x,y,v,w)f^\text{in}(y,w)\dd y\dd w \\
+ \int_0^t\iint_{\RR^2}b_{xw}(t-s)\partial_w\partial_vG(t-s,x,y,v,w)F_f(s,y)\partial_wf(s,y,w)\dd y\dd w\dd s\\
+ \int_0^{t/2}\iint_{\RR^2}b_{xy}(t-s)\partial_v\partial_yG(t-s,x,y,v,w)F_{f}(s,y)\partial_wf(s,y,w)\dd y\dd w\dd s\\
+ \int_{t/2}^t\iint_{\RR^2}b_{xy}(t-s)\partial_vG(t-s,x,y,v,w)F_{\partial_yf}(s,y)\partial_wf(s,y,w)\dd y\dd w\dd s\\
+ \int_{t/2}^t\iint_{\RR^2}b_{xy}(t-s)\partial_vG(t-s,x,y,v,w)F_f(s,y)\partial^2_{yw}f(s,y,w)\dd y\dd w\dd s.
\end{multline*}
Notice the splitting of the time integral and the appropriate distribution of the $\partial_y$ derivative to ensure integrability of singularities in time. The reasoning can be continued for higher order derivatives and one obtains that $\partial_v^n\partial_x^mf(t)$ has the same singularity as $\partial_v^n\partial_x^mG(t)$ when $t\to0$.
\end{proof}

\begin{cor}\label{cor:regu}
Under the assumptions of Proposition~\ref{prop:mildsol_nonlinVFP}, any mild solution \eqref{eq:VFPnonlin} is a smooth classical solution  $f\in\mathcal{C}([0,\infty), L^1(\RR^2))\cap \mathcal{C}^{\infty}((0,\infty)\times\RR^2)$ satisfying the first two equations of \eqref{eq:VFPnonlin} point-wise for all $t>0$ and $(x,y)\in\RR^2$ and the last equation a.e. $(x,y)\in\RR^2$.
\end{cor}

\subsection{Proof of Theorem~\ref{theo:well}}\label{sec:proof_theo_1}
The proof of Theorem~\ref{theo:well} is completed by combining Proposition~\ref{prop:mildsol_nonlinVFP}, Proposition~\ref{prop:regu} and  Corollary~\ref{cor:regu}.

\subsection{Additional properties}

We end Section~\ref{sec:cauchy} with additional properties on the Vlasov-Fokker-Planck equation. 

\begin{lem}[$L^p$ norms] \label{lem:Lp}
For any $T>0$ and $p\in[1,+\infty]$ one has 
\[
\|f\|_{L^\infty(0,T;L^p(\RR^2))}\ \leq\ \|f^\text{in}\|_{L^p(\RR^2)}e^{2\nu(1-1/p)T}
\]
\end{lem}
\begin{proof}
Let $p<+\infty$. One has
\[
\begin{aligned}
\frac{\dd}{\dd t}\|f(t)\|_{L^p(\RR^2)}^p &= \iint_{\RR^2}\left(-\partial_x(vf^p) + \partial_v((F_f(t,x) -\alpha x)f^p)+2\nu p f^{p-1}\partial_v(vf+\partial_vf)\right)\dd v\dd x\\&= -\iint_{\RR^2}2\nu p(p-1) f^{p-2}\partial_vf(vf+\theta\partial_vf)\dd v\dd x\\
&\leq -\iint_{\RR^2}2\nu (p-1) v\partial_vf^p\dd v\dd x = 2\nu (p-1)\|f(t)\|_{L^p(\RR^2)}^p.
\end{aligned}
\]
The claim is obtained after a Grönwall lemma. The case $p=+\infty$ is obtained as a limit.
\end{proof}

In the next lemma we recall the continuity equation.

\begin{lem}[Continuity equation]\label{lem:conteq}
It holds that
\begin{equation}\label{eq:conteq}
\partial_t\rho + \partial_x j = 0, \quad\text{where}\quad j(t,x)
 = \int_\RR v f(t,x,v)\dd v.
\end{equation}
\end{lem}
\begin{proof}
Integrate \eqref{eq:VFPnonlin} in the $v$ variable.
\end{proof}

\section{Steady states and long-time behavior}\label{sec:long}

\subsection{Haissinski solutions}
In this section we investigate the existence and uniqueness of stationary solutions to equation \eqref{eq:VFPnonlin} of the form
\begin{equation} \label{def:haissinski} 
f^\infty(x,v) = \frac{\mathcal{M}}{\sqrt{2\pi\theta}}e^{-\frac{v^2}{2\theta}}\sigma^\infty(x)\,,
\end{equation} 
where $\mathcal{M}>0$ denotes the mass and $\sigma^\infty$ is normalized, namely
\[
\frac{1}{\mathcal{M}}\iint_{\RR^2} f^\infty(x,v) \dd v\dd x = \int_\RR\sigma^\infty(x)\dd x = 1.
\]
In order to solve \eqref{eq:VFPnonlin}, one observes that $\sigma^\infty$  should be a fixed point of the mapping
\begin{equation}\label{eq:HaissinskiOperatorT}
\cT(\sigma)(x):=\dfrac{e^{-\frac{\alpha x^2}{2\theta}- \frac{I \mathcal{M}}{\theta}K\ast\sigma(x)}}{\int_\RR e^{-\frac{\alpha y^2}{2\theta} - \frac{I \mathcal{M}}{\theta}K\ast\sigma(y)}\dd y}.
\end{equation}
Let us introduce the Banach spaces 
\[
\mathcal{X} = \{\sigma\in \mathcal{C}(\RR)\ \text{such that}\ \sup_{x\in\RR}|\sigma(x)|e^{\frac{\beta x^2}{2}}<+\infty\}
\]
with $0<\beta<\alpha/\theta$ endowed with the norm
\[
\|\sigma\|_{\mathcal{X}} = \sup_{x\in\RR}|\sigma(x)|\,e^{\frac{\beta x^2}{2}}.
\]
We use Schaefer's fixed point theorem to show that $\cT$ has a fixed point.

\begin{prop}\label{prop:existence_steady}
Assume that $K\in L^\infty$ and  $\partial_xK\in L^p$ for some $p\in[1,\infty]$. %\todo{Here there is probably room to optimize the result in terms of functional spaces.}. 
 Then there exists a fixed point $\sigma^\infty$ to the map $\cT:\mathcal{X}\to\mathcal{X}$. Moreover, $\sigma^\infty$ is positive, has integral $\int_\RR\sigma^\infty=1$ and $\partial_x^n\sigma^\infty\in\cX$ for all $n\geq0$.
\end{prop}
\begin{proof}
 To begin with, observe that thanks to Young's convolution inequality and Hölder inequality,  for any $G\in L^p$, one has the estimate
\[
\|G\ast\sigma\|_{L^\infty}\leq \|G\|_{L^p}\|\sigma\|_{L^{p'}}\leq C_{p} \|G\|_{L^p}\|\sigma\|_{\cX}
\]
with a finite $C_p >0$ and $p'=(1-p^{-1})^{-1}$. Then we complete the proof in three steps.

\noindent\emph{Step 1: $\cT$ is compact.}\\ Let $R>0$ such that $\|\sigma\|_{\mathcal{X}}\leq R$. Then for $\varepsilon = (\alpha/\theta)-\beta>0$.
\[
0< \cT(\sigma)(x)e^{\frac{\beta x^2}{2}}\leq (2\pi\theta/\alpha)^{\frac12}e^{2I\mathcal{M}\|K\|_{L^\infty}\|\sigma\|_{L^1}-\frac{\varepsilon x^2}{2}}\leq (2\pi\theta/\alpha)^{\frac12}e^{2I\mathcal{M}C_1\|K\|_{L^\infty}R-\frac{\varepsilon x^2}{2}}
\]
and
\[
|\partial_x\cT(\sigma)(x)|e^{\frac{\beta x^2}{2}}\leq (2\pi\theta/\alpha)^{\frac12}e^{2I\mathcal{M}C_1\|K\|_{L^\infty}R}(|x|+I\mathcal{M}C_p\|\partial_xK\|_{L^p} R)e^{-\frac{\varepsilon x^2}{2}}
\]
In particular $\cT(\sigma)\in\cX$ and there is $C_R>0$ depending on  $\|K\|_{L^\infty}$, $\|\partial_xK\|_{L^p}$ and the parameters such that
\[
\cT(\sigma)\in Y_R := \{\nu\in\cX, \|\nu\|_{\cX} + \|\partial_x\nu\|_{\cX}\leq C_R\}.
\]
Let $(\nu_n)_n$ be a sequence in $Y_R$ and for any positive integer $m$ define $K_m = [-m,m]$. Then $(\nu_n|_{K_m})_n$ is equicontinuous, so by Arzela-Ascoli and a diagonal argument  one can find a subsequence $(n_k)_k$ and a continuous function $\nu$ such that for all $m\geq1$ 
\[
\|\nu_{n_k}-\nu\|_{\mathcal{C}(K_m)}\to 0\,,\text{ as }k\to\infty
\]
but since the sequence and limit belong to a ball of $\cX$ of radius at most $C_R$, one has
\[
\|\nu_{n_k}-\nu\|_{\cX}\leq \|(\nu_{n_k}-\nu)e^{\beta x^2/2}\|_{\mathcal{C}(K_m)}+2C_Re^{-\varepsilon m^2/2}
\]
so one can take $m$ large enough and then $k$ large enough to make the r.h.s. as small as desired. It follows that $Y_R$ is sequentially compact, and thus $\cT$ is compact. 

\noindent\emph{Step 2: Schaefer's condition.}\\ Let $\lambda\in[0,1]$ and consider $\sigma\in\cX$ such that 
\[
\lambda \cT(\sigma) = \sigma.
\]
Then $\|\sigma\|_{L^1} = \lambda\leq 1$ therefore
\[
\|\sigma\|_{\cX}\leq\|\cT(\sigma)\|_{\cX}\leq(2\pi\theta/\alpha)^{\frac12}e^{2I\mathcal{M}\|K\|_{L^\infty}}
\]
so $\cT$ has a fixed point $\sigma^\infty$.

\noindent\emph{Step 3: Properties of fixed points $\sigma^\infty$.}\\ First, since $\sigma^\infty = \cT(\sigma^\infty)$, $\sigma^\infty$ is positive and has an integral equal to $1$. Then, observe that by successive differentiation that \[\partial_x^{n+1}\sigma^\infty(x) = P(x,{\partial_xK\ast}\sigma^\infty,\dots,{\partial_xK\ast}\partial_x^n\sigma^\infty)\cT(\sigma^\infty)\] for some multivariate polynomial $P(X_1, X_2, \dots, X_{n+1})$. We proceed by induction. Assume that $\partial_x^k\sigma^\infty\in\cX$ for all $0\leq k\leq n$, then for any polynomial, \[x\to e^{-\frac{\beta x^2}{2}}P(x,{\partial_xK\ast}\sigma^\infty(x),\dots,{\partial_xK\ast}\partial_x^n\sigma^\infty(x))\in L^\infty.\] Moreover $x\mapsto \cT(\sigma^\infty(x))e^{\frac{\beta x^2}{2}}\in\cX$ and therefore $\partial_x^{n+1}\sigma^\infty\in\cX$.
\end{proof}

\begin{prop}[Uniqueness at low current]\label{prop:uniqueness_steady}
If the current intensity is such that
\[
I < I^\text{thres}:=\frac{C \theta}{\mathcal{M}\|K\|_{L^\infty}}
\]
with $C>0$ a universal constant, then $\cT: L^1\to L^1$ has at most one fixed point.
\end{prop}
\begin{proof}
From the proof of Proposition~\ref{prop:existence_steady}, it is clear that $\cT: L^1\to L^1$ is a well-defined. Let $\sigma_1,\sigma_2\in L^1$ be two fixed points. Then one obtains
\[
\|\sigma_1-\sigma_2\|_{L^1}\leq \frac{\iint_{\RR^2}e^{-\frac{\alpha(x^2+y^2)}{2\theta}}\left|e^{-\frac{I\mathcal{M}}{\theta}(K\ast\sigma_1(x) + K\ast\sigma_2(y))} - e^{-\frac{I\mathcal{M}}{\theta}(K\ast\sigma_2(x) + K\ast\sigma_1(y))}\right|\dd x\dd y}{\iint_{\RR^2}e^{-\frac{\alpha(x^2+y^2)}{2\theta}-\frac{I\mathcal{M}}{\theta}(K\ast\sigma_1(x) + K\ast\sigma_2(y))}\dd x\dd y},
\]
which yields, since $\|K\ast\sigma\|_{L^\infty}\leq \|K\|_{L^\infty}\|\sigma\|_{L^1}$, that
\[
\|\sigma_1-\sigma_2\|_{L^1}\leq e^{\frac{3 I \mathcal{M}}{\theta}\|K\|_{L^\infty}}\tfrac{2I\mathcal{M}}{\theta}\|K\|_{L^\infty}\|\sigma_1-\sigma_2\|_{L^1}.
\]
The contraction property then applies as soon as the condition of the statement is satisfied with $C = \tfrac13 W(\tfrac 32)\simeq 0.24$, where $W$ is the Lambert function. 
\end{proof}
{We end this section with upper bounds on the potential associated with the density $\sigma^\infty$.% and related quantities which will be useful later.

\begin{lem}\label{lem:constVinf}
The quantity $V^\infty = -\ln(\sigma^\infty)$ %and $W = (1+|\partial_xV^\infty|^2)^{1/2}$ 
can be estimated by
\[
|\partial_x^2 V^\infty(x)|\leq C^\infty = \tfrac{2\alpha}{\theta} + \left(\tfrac{I\mathcal{M}}{\theta}\right)^2\|\partial_x K\|_{L^\infty}^2,
\]
%and 
%\[
%\int W^2\sigma^\infty \leq 1 + C^\infty.
%\]
\end{lem}
\begin{proof}
Using $V^\infty = -\ln(\cT(\sigma^\infty))$ and Young's convolution inequality one obtains
\[
\|\partial_x^2 V^\infty\|_{L^\infty}\leq \tfrac{\alpha}{\theta} + \tfrac{I\mathcal{M}}{\theta}\|\partial_x K\|_{L^\infty}\|\partial_x \sigma^\infty\|_{L^1}.
\]
Young's inequality then yields
\[
\|\partial_x \sigma^\infty\|_{L^1} = \int |\partial_x V^\infty|e^{-V^\infty}\leq \frac12 \int \left(\tfrac{I\mathcal{M}}{\theta}\|\partial_x K\|_{L^\infty} + \left(\tfrac{I\mathcal{M}}{\theta}\|\partial_x K\|_{L^\infty}\right)^{-1}|\partial_x V^\infty|^2\right)e^{-V^\infty}
\]
where
\[
\int|\partial_x V^\infty|^2e^{-V^\infty} = \int\partial_x^2 V^\infty e^{-V^\infty}\leq \|\partial_x^2 V^\infty\|_{L^\infty}.
\]
Combining everything yield the result.
\end{proof}

%
%\begin{rema}[Extensions]
%For Proposition~\ref{} The Banach space could of course be replaced by
%\[
%\mathcal{X}_w = \{\sigma\in \mathcal{C}(\RR)\ \text{such that}\ \sup_{x\in\RR}|\sigma(x)w(x)<+\infty\}
%\]
%with $w$ any continuous function such that $w(x)e^{-\alpha x^2/2}\to 0$ as $x\to\infty$.
%\end{rema}
\subsection{Proof of Theorem~\ref{theo:steady}}\label{sec:proof_theo_2}
The proof of Theorem~\ref{theo:steady} is obtained by combining the results of Proposition~\ref{prop:existence_steady} and Proposition~\ref{prop:uniqueness_steady}.

\subsection{Free energy estimate and symmetry of the interaction kernel}\label{sec:freeenerg}

Next we derive a so-called free energy (or entropy) estimate for the solutions of \eqref{eq:VFPnonlin}. Let us precise that here, free energy does not refer to a particular physical quantity in the context of particle accelerators. The name is chosen by mathematical analogy with global quantities arising in other models. Let us first decompose the interaction potential into an even and an odd part, namely
\begin{equation}\label{eq:oddeven}
K^\text{e}(x) = \frac{K(x) + K(-x)}{2}\,, \quad K^\text{o}(x) = \frac{K(x) - K(-x)}{2}\,.
\end{equation}
We define
\begin{multline}\label{eq:entropy}
\cE_{\blue{f}}(t) = \theta\iint_{\RR^2}f\log f\dd v\dd x + \iint_{\RR^2}\frac{v^2 + \alpha x^2}{2}f\dd v\dd x\\ + \frac{I}{2}\iint_{\RR^2}K(x-y)\rho(t,x)\rho(t,y)\dd x\dd y.
\end{multline}
\begin{rema}
Observe that for the last term, \blue{$K$ can be replaced by $K^\text{e}$} without changing the value of the integral.
\end{rema}

\begin{prop}[Entropy estimate]\label{prop:freeenerg}
One has 
\[
\frac{\dd \cE_{\blue{f}}(t)}{\dd t} + 2\nu\iint_{\RR^2} \frac{1}{f}|vf+\theta\partial_vf|^2\dd v\dd x\ =\ - I\iint\partial_xK^\text{o}(x-y)j(t,x) \rho(t,y)\dd x\dd y.
\]
where $j$ is defined in \eqref{eq:conteq}.
\end{prop}
\begin{proof}
On the one hand,
\[
\begin{aligned}
&\frac{\dd}{\dd t}\left(\theta\iint_{\RR^2}f\log f\dd v\dd x + \iint\frac{v^2 + \alpha x^2}{2}f\dd v\dd x\right)\\
=&\iint_{\RR^2}\Big(-\partial_x(vf) - \partial_v((F_f - \alpha x)f)+2\nu\partial_v(vf+\theta\partial_vf)\Big)\left(\theta\log(f) + \theta + \frac{v^2}{2} + \frac{\alpha x^2}{2} \right)\dd v\dd x\\
=&\iint_{\RR^2}\left(\alpha xvf + (F_f - \alpha x)vf)\right)\dd v\dd x
-2\nu\iint_{\RR^2}\frac{1}{f}|vf+\theta\partial_vf|^2\dd v\dd x\\
=&I\iint_{\RR^2}K(x-y)\partial_xj(t,x) \rho(t,y)\dd x\dd y
-2\nu\iint_{\RR^2}\frac{1}{f}|vf+\theta\partial_vf|^2\dd v\dd x.\\
\end{aligned}
\]
On the other hand, since $K^\text{e}$ is even, one has 
\[
\begin{aligned}
\frac{\dd}{\dd t}\left(\frac{I}{2}\iint_{\RR^2}K^{\text{e}}(x-y)\rho(t,x)\rho(t,y)\dd x\dd y\right)
=& I\iint_{\RR^2}K^{\text{e}}(x-y)\partial_t\rho(t,x)\rho(t,y)\dd x\dd y\\
=& -I\iint_{\RR^2}K^{\text{e}}(x-y)\partial_xj(t,x)\rho(t,y)\dd x\dd y\,,
\end{aligned}
\]
where we used the continuity equation \eqref{eq:conteq}. By summing the two identities one obtains the result.
\end{proof}

The equality of Proposition~\ref{prop:freeenerg} shows that $\cE_{\blue{f}}$ is a Lyapunov functional only when $K$ is even. More precisely one has the following result. 

\begin{cor}\label{cor:notLyapunov}
If $K$ is even, then 
\[
\frac{\dd \cE_{\blue{f}}(t)}{\dd t}\leq 0,\quad \forall t\geq 0
\]
Conversely, if $K$ is not even function, there is $f^\text{in}$ such that 
\[
\left.\frac{\dd \cE_{\blue{f}}(t)}{\dd t}\right|_{t=0} > 0.
\]
\end{cor}
\begin{proof}
If $K$ is even, the decay of $\mathcal{E}_{\blue{f}}$ is an immediate consequence of Proposition~\ref{prop:freeenerg}. If $K$ is not even then $\partial_xK^\text{o}\neq0$ and one can find a non-zero, non-negative integrable $\rho^\text{in}$ such that \blue{$-\iint  \partial_xK^\text{o}(x-y)j^\text{in}(x)\rho^\text{in}(y)\dd x\dd y > 0$. One can take for instance $f^\text{in} = e^{-(v-v_*)^2/(2\theta)}\rho^\text{in}(x)$ with $\text{sign}(v_*) = - \text{sign}(\iint \partial_x K^\text{o}(x-y)\rho^\text{in}(x)\rho^\text{in}(y)\dd x\dd y)$.} Then, in the entropy balance of Proposition~\ref{prop:freeenerg} the second term of the left hand side is positive and $O(v_*^2)$  and the first term of the right hand side is positive and $O(v_*)$. Thus for a small enough $|v_*|$ the entropy increases initially.
\end{proof}
\blue{
\begin{rema}
One can inspect the critical points of the Lagrangian functional 
\[
\mathcal{L}(f,\lambda) = \cE_{f} + \lambda \left(\iint_{\RR^2} f\dd v\dd x-\cM\right)
\]
corresponding to the problem of minimizing $\cE_{f}$ under the constraint of given mass $\cM>0$. The Gateaux derivative of the Lagrangian in the $f$ variable is given by
\begin{align*}
    \lim_{\eps\to 0}\frac{\mathcal{L}(f+\eps\varphi,\lambda) - \mathcal{L}(f,\lambda)}{\eps} = \iint_{\RR^2} \varphi\left( \theta\log(f) +\theta +\lambda +\frac{v^2+\alpha x}{2} + I K^\text{e}\ast \rho(x) \right) \dd x \dd v.
\end{align*}
For the right-hand side to vanish (for all test function $\varphi$) and $\partial_\lambda\mathcal{L}(f,\lambda) = 0$, the function $f$ must be of the form \eqref{def:haissinski} with $\sigma^\infty$ coinciding with Haissinki solutions iff $K=K^\text{e}$. 
\end{rema}
To end this section on free energy estimates, we point out \cite{monmarche2023note}, which came out after the initial version of the present paper and provides interesting complementary results on \eqref{eq:VFPnonlin}. In particular, under the assumption that $\partial_xK$ is Lipschitz continuous, it is shown that one can design a Lyapunov functional for the system, even for non-symmetric interaction kernels. Unlike the candidate functional \eqref{eq:entropy}, it is in general a non-explicit quantity which is obtained as a limit of the relative entropy for the corresponding many particle system (see \cite[Proposition~2]{monmarche2023note}).
}

\subsection{Hypocoercivity estimates}

In this section we investigate the convergence of a solution to \eqref{eq:VFPnonlin} to the Haissinski solution $f^\infty$ of Theorem~\ref{theo:steady}. To that end we decompose the solution $f$ to \eqref{eq:VFPnonlin} as 
\begin{align}\label{eq:defg}
	f = f^\infty + f^\infty g. 
\end{align}
The perturbation $g$ is solution to 
\begin{equation} \label{eq:hypoVFPg}
\left\{ \begin{aligned} & \pa_t g + v\pa_x g + (F_{f^\infty} - \alpha x) \pa_v g - \frac{v}{\theta} F_{f^\infty g} = 2\nu\theta \left(\pa_v - \frac{v}{\theta}\right) \pa_v g + Q(g,g), \\
							    	&Q(g,g) = F_{f^\infty g} \left(\frac{v}{\theta} - \pa_v \right) g . \\
							    	& g _{\lvert_{t=0}} = g^\text{in} := \frac{f^\text{in}}{f^\infty} - 1. 
		 \end{aligned}\right. 
\end{equation}	
In the Hilbert space $\cH = \lbrace g\in L^2(f^\infty \dd x \dd v), \iint_{\RR\times\RR} gf^{\infty} \dd x \dd v = 0 \rbrace$ we denote the norm by $\|\cdot \|$ and the scalar product by $\langle \cdot , \cdot \rangle$. We introduce the transport operator 
\begin{align} \label{eq:hypoT}
	T = v\pa_x + (F_{f^\infty} -\alpha x)\pa_v 
\end{align}
which is antisymmetric in $\cH$, and recognise the adjoint $\pa_v^\ast = \frac{v}{\theta} - \pa_v$ hence we define the collision operator 
\begin{align}\label{eq:hypoL}
	L = - 2\nu \theta \pa_v^\ast  \pa_v 
\end{align}
which is symmetric in $\cH$. The equation then reads
\begin{align} \label{eq:hypoVFPcan}
\pa_t g + Tg - \frac{v}{\theta} F_{f^\infty g} = Lg + Q(g,g). 
\end{align}
In the framework of  hypocoercivity methods of \cite{DMS} (see also \cite{ADLT} for an application to the Vlasov-Poisson-Fokker-Planck system), we introduce $\Pi$ the orthogonal projection onto the null space of $L$, in our case 
\[ \Pi g = \frac{\rho_g}{\mathcal{M}\sigma^\infty}  \mbox{ with } \rho_g = \int g f^\infty \dd v .\]
In other words, with the Gaussian $M_{\theta} = \frac{1}{\sqrt{2\pi\theta}} e^{-v^2/2\theta}$ we have %$\rho_g = \int g M_\theta \dd v$. 
{$\Pi g  = \int g M_\theta \dd v$.} 
We also introduce the operator $A$ given by 
\[ A = \big(\Id + (T\Pi)^\ast T\Pi\big)^{-1} (T\Pi)^\ast.\]
\subsubsection{Linear hypocoercivity} Classical hypocoercivity theory then states that if the following four assumptions are satisfied : 
\begin{itemize}
	\item[(H1)] Microscopic coercivity: $\exists \lambda_m >0$ s.t. $\forall h\in\mathcal{D}(L)$ 
	\[ - \langle Lh,h \rangle \geq \lambda_m \| (\Id -\Pi) h\|^2 \]
	\item[(H2)] Macroscopic coercivity: $\exists \lambda_M >0$ s.t.  $\forall h\in\mathcal{H}$, $\Pi h \in \mathcal{D}(T)$ 
	\[ \| T\Pi h \|^2 \geq \lambda_M \| \Pi h \|^2 \]
	\item[(H3)] Parabolic macroscopic dynamics: $\forall h\in\mathcal{H}$ 
	\[ \Pi T \Pi h = 0 \]
	\item[(H4)] Bounded auxiliary operators 
	\[ \|AT(\Id-\Pi)h \| + \|ALh\| \leq C_M \| (\Id -\Pi) h\| \]
\end{itemize}
then the solution $h$ to the linear equation 
\begin{equation} \label{eq:linVFP}
	\pa_t h + Th = Lh 
\end{equation}
converges exponentially fast to the steady state in $\cH$. The proof relies on the modified entropy functional 
\begin{equation} \label{def:hypoLyap}
	H[h] := \frac12 \| h \|^2 + \eps \langle Ah, h \rangle .
\end{equation} 
Under the assumptions above and for $\eps$ small enough one can show %-- see e.g. \cite[Proposition 4]{ADLT} -- 
that $H$ is the square of a norm which is equivalent to $\| \cdot \|$: 
\[ {\frac{1-\eps}{2}} \| h\|^2 \leq H[h] \leq {\frac{1+\eps}{2}} \| h \|^2, \quad \forall h\in \cH. \]
If $h$ is solution to the linear problem \eqref{eq:linVFP} then 
\begin{equation} \label{def:dissipation} 
\begin{aligned} 
	\frac{\dd}{\dd t} H[h(t)] := - \mathcal{D}_\eps [h(t)] &=  \langle Lh,h\rangle - \eps \langle AT\Pi h,h \rangle + \eps \langle TAh,h\rangle \\
	&- \eps \langle AT(\Id -\Pi)h,h \rangle + \eps \langle ALh,h \rangle 
\end{aligned}
\end{equation}
where $\mathcal{D}_\eps$ is called the dissipation of entropy functional. The assumptions (H1)-(H4) allow for a control of each term on the right hand side, see \cite[Theorem 2]{DMS}, which yields 
\begin{multline} \label{eq:hypocoef} 
	\frac{\dd}{\dd t} H[h(t)] \leq \langle Lh,h\rangle  - \eps \frac{\lambda_M}{1+\lambda_M} \| \Pi h \|^2+ \eps (1+C_M) \| (\Id -\Pi)h\| \|h\| \\
	\leq {- \lambda_m \| (\Id -\Pi)h\|^2} - \eps \frac{\lambda_M}{1+\lambda_M} \| \Pi h \|^2+ \eps (1+C_M) \| (\Id -\Pi)h\| \|h\| 
\end{multline} 	
Choosing $\eps$ small enough one can find ${\lambda_\text{lin}}>0$ depending on $\eps$, $\lambda_m$, $\lambda_M$ and $C_M$ such that
\[ \frac{\dd}{\dd t} H[h(t)] \leq - {\lambda_\text{lin}} H[h(t)] \]
and the exponential decay towards the steady-state follows. 

In order to derive an hypocoercivity estimate for our nonlinear equation \eqref{eq:hypoVFPcan} we first prove that the assumption (H1)-(H4) are satisfied.
\begin{lem}\label{lem:LinHypo}
The operators $T$ and $L$ defined in \eqref{eq:hypoT} and \eqref{eq:hypoL} satisfy the assumptions (H1)-(H4) {with the constants 
\[
\lambda_m = 2\nu\,,\quad \lambda_M = \frac{\alpha}{\theta}e^{-4\frac{I\mathcal{M}}{\theta}\|K\|_{L^\infty}}\,,\quad C_M = \nu + 4\theta^2+4\alpha\theta + 2\left(I\mathcal{M}\|\partial_x K\|_{L^\infty}\right)^2.
\]}
\end{lem}
\begin{proof}
Assumption (H1) follows from the Gaussian Poincaré inequality in $\cH$: 
\begin{align*}
	-\langle Lh,h \rangle = 2\nu\theta \| \pa_v h \|^2 \geq {2\nu} \| (\Id-\Pi) h \|^2 .
\end{align*}
For assumption (H2) we notice that $T\Pi h = v\pa_x \Pi h$ and $\int \Pi h \sigma^\infty = 0$. Moreover, since $K\ast \sigma^\infty \in L^\infty(\RR)$, the Holley-Stroock perturbation property (see e.g. \cite[Section 5.1.2]{BackryGentilLedoux} ensures a Poincaré inequality. More precisely, {using the fact that $\sigma^\infty = \cT(\sigma^\infty)$ and the Gaussian Poincaré inequality we get
\begin{align*}
	\| T\Pi h \|^2 = \cM \int_\RR |\pa_x \Pi h|^2 \sigma^\infty \dd x \geq \lambda_M  \cM \int_\RR |\Pi h|^2 \sigma^\infty \dd x = \lambda_M \| \Pi h \|^2,
\end{align*}
with $\lambda_M$ as in the statement.} Since the velocity profile of $f^\infty$ is a centered Gaussian we have immediately assumption (H3): $\Pi T \Pi = 0$ by symmetry. Finally, for assumption (H4), we begin with the control of $\| ALh\|$. One can easily check that 
\[ (T\Pi)^\ast h = - (\Pi T) h = - \left[ \pa_x  + {\theta^{-1}}(F_{f^\infty}-\alpha x) \right] \Pi(vh) = \pa_x^\ast \Pi (vh) \]
and 
\[ \Pi (vLh) = -2\nu \theta \int v \pa_v^\ast \pa_v h M_\theta \dd v = -2 \nu \Pi (vh)  \]
hence 
\[ (T\Pi )^\ast Lh = \pa_x^\ast \left( -2\nu \Pi(vh) \right) = -2\nu (T\Pi )^\ast h.\]
Moreover, using \cite[Lemma 1]{DMS} we know that assumptions (H1)-(H3) yield 
\begin{equation}\label{eq:controlLemma1DMS}
 \| Ah\| \leq \frac12 \| (\Id -\Pi)h \|, \quad \| TAh \| \leq \| (\Id -\Pi)h \|
\end{equation}
hence 
\begin{align*}  
	\| ALh \| &= \| \big(\Id +(T\Pi)^\ast T\Pi\big)^{-1}(T\Pi )^\ast L h \| = 2\nu \| \big(\Id +(T\Pi)^\ast T\Pi \big)^{-1}(T\Pi )^\ast h\| \\
	&= 2\nu \| Ah \| \\
	&\leq \nu \| (\Id -\Pi) h \|. 
\end{align*}
For the control of $\|AT(\Id -\Pi)h\|$ we work with the adjoint 
\[ [AT(\Id -\Pi)]^\ast = -(\Id -\Pi)T^2 \Pi [\Id +(T\Pi)^\ast T\Pi ]^{-1} \]
where 
\[ (T\Pi)^\ast T\Pi h = \pa_x^\ast \Pi \left( v^2 \pa_x \Pi h \right) = \theta \pa_x^\ast \pa_x \Pi h \]
and 
\[ (\Id -\Pi)T^2 \Pi h= (\Id -\Pi) \left( {v^2\pa_x^2} \Pi h + (F-\alpha x)\pa_x \Pi h \right) = (v^2-\theta)\pa_x^2\Pi h \]
hence, 
\[ [AT(\Id -\Pi)]^\ast h = (\theta-v^2)\pa_x^2 (\Id +\theta \pa_x^\ast \pa_x )^{-1} \Pi h. \]
Let us introduce $u = (\Id +\theta \pa_x^\ast \pa_x )^{-1} \Pi h$, i.e. $u$ solution to 
\begin{align} \label{eq:hypoH4}
	\sigma^\infty u - \pa_x (\sigma^\infty \pa_x u ) = \sigma^\infty \Pi h .
\end{align}
One can easily check using Proposition \ref{prop:existence_steady} that the steady-state $\sigma^\infty$ satisfies the assumptions of \cite[Proposition 5]{DMS} which ensure an elliptic regularity estimate for the solutions of \eqref{eq:hypoH4}. {In order to derive explicit constants here we briefly redo the argument. Clearly from \eqref{eq:hypoH4}, one gets $\|u\| \leq \|\Pi h\|$ and  $2\|\partial_xu| \leq \|\Pi h\|$ from integrating against $u$. Similarly, and using the previous bound $\|(\sigma^\infty)^{-1}\pa_x (\sigma^\infty \pa_x u )\|\leq 2\|\Pi h\|$ by integrating against $(\sigma^\infty)^{-1}\pa_x (\sigma^\infty \pa_x u )$. Then by expanding the square $\|(\sigma^\infty)^{-1}\pa_x (\sigma^\infty \pa_x u )\|^2 = \|\partial_xV^\infty\pa_xu\|^2 -\langle\partial_xV^\infty, \pa_x(|\pa_xu|^2)\rangle + \|\pa^2_xu\|^2$ and integrating the middle term by parts one concludes
\[
\| \pa_x^2 u \|^2 \leq (2+C^\infty)\|\Pi h\|^2,
\]
with $C^\infty$ the bound of Lemma~\ref{lem:constVinf}.} As a result, using the fact that $\pa_x^2 u$ does not depend on the variable $v$, we have
\begin{align*}
	\|  [AT(\Id -\Pi)]^\ast h \|^2 &\leq \cM \int_\RR (\theta - v^2)^2 M_\theta(v) \dd v \int_\RR |\pa_x u|^2 \sigma^\infty \dd x \\
	&\leq  %\cM 2\theta^2 C_{\text{ell}} 
	{2\theta^2(2+C^\infty)}\| \Pi h \|^2 .
\end{align*}
Finally, since $\Pi [AT(\Id -\Pi)]^\ast h = \int_\RR (\theta-v^2) M_\theta \dd v \, \pa_x^2 (\Id +\theta \pa_x^\ast \pa_x )^{-1} \Pi h = 0 $ we conclude that for any $h, g \in \cH $
\begin{align*}
    \left| \langle AT(\Id -\Pi) h , g \rangle \right| &= \left| \langle h, [AT(\Id -\Pi)]^\ast g\rangle \right| \\
    &= \left| \langle (\Id -\Pi) h , [AT(\Id -\Pi)]^\ast g\rangle \right| \\
    &\leq {2\theta^2(2+C^\infty)}\| (\Id -\Pi) h \| \| g \| 
\end{align*}
and (H4) follows. 
\end{proof} 

\subsubsection{Nonlinear hypocoercivity}  Let us now turn to the nonlinear system \eqref{eq:hypoVFPg}. 
{\begin{prop} \label{prop:hypocoercivity}
	Let $g$ solve \eqref{eq:hypoVFPg}. There are $\lambda, C, \varepsilon_0>0$ depending only on $\alpha,\theta, \nu, \|K\|_{W^{1,\infty}}$ such that if $\mathcal{M}I < C$ and $\eps<\eps_0$, then 
	\begin{align} \label{eq:HypoGronwall}
		\frac12 \frac{\dd}{\dd t} H[g](t) \leq - \lambda H[g](t)+ C \|\Pi g\|_{L^1(\RR)}^2 H[g](t). 
	\end{align}
\end{prop}}
\begin{proof}
For $g$ a solution to \eqref{eq:hypoVFPg}, differentiating the modified entropy function $H$ of \eqref{def:hypoLyap} yields
\begin{equation} \label{eq:hypoNLH} 
	\begin{aligned}
	\frac{\dd}{\dd t} H[g](t) &= -\mathcal{D}_\eps [g(t)] + \langle g, \frac{v}{\theta} F_{f^\infty g}\rangle + \langle g, Q(g,g)\rangle +  \eps \langle A\frac{v}{\theta}F_{f^\infty g}, g\rangle \\
	&\quad + \eps \langle Ag, \frac{v}{\theta} F_{f^\infty g}\rangle \quad + \eps \langle AQ(g,g),g\rangle + \eps \langle Ag, Q(g,g) \rangle. \\
	&= -\mathcal{D}_\eps [g(t)] + \langle g, \frac{v}{\theta} F_{f^\infty g}\rangle {+ \langle g, Q(g,g)\rangle} +  \eps \langle A\frac{v}{\theta}F_{f^\infty g}, \Pi g\rangle\\& \quad + \eps \langle AQ(g,g),\Pi g\rangle . 
\end{aligned}
\end{equation} 
{Indeed since $Ag = \Pi A g $ and $F_{f^\infty g} = \Pi F_{f^\infty g}$, one has for the  \[\langle Ag, \frac{v}{\theta} F_{f^\infty g}\rangle = 0,\] as the first moment of a centered Gaussian is zero. Furthermore, since $Q(g,g)=\pa_v^\ast(F_{f^\infty g}  g)$ we also have \[\langle Ag, Q(g,g) \rangle =\langle \partial_v\Pi Ag, F_{f^\infty g}  g \rangle= 0.\]}

We already know the control \eqref{eq:hypocoef} of $\mathcal{D}_\eps[g](t)$.  For the following term we have
{\begin{align*}
\langle g, \frac{v}{\theta} F_{f^\infty g} \rangle &= \langle (\Id -\Pi)g, \frac{v}{\theta} F_{f^\infty g} \rangle\\
&\leq \|F_{f^\infty g}\|_{L^\infty} \|\frac{v}{\theta}\| \|(\Id -\Pi)g\|\\
&\leq I\sqrt{\frac{\mathcal{M}}{\theta}}\|\pa_x K\|_{L^\infty}\|\rho_g\|_{L^1_x}\|(\Id -\Pi)g\|\\ 
&\leq I\frac{\mathcal{M}}{\sqrt{\theta}}\|\pa_x K\|_{L^\infty} \| \Pi g \| \| (\Id -\Pi)g \|. 
\end{align*}}

{For the third term of the right-hand side, one has  
\begin{align*}
	\langle Q(g,g) ,g \rangle &= \langle F_{f^\infty g}  g  ,\partial_vg \rangle \\
	&\leq \|F_{f^\infty g}\|_{L^\infty}\|g\|\|\partial_vg\|\\
	&\leq I\|\pa_x K\|_{L^\infty}\|\rho_g\|_{L^1_x}\|g\|\|\partial_vg\|
\end{align*}}

{Then we can control the fourth term using \eqref{eq:controlLemma1DMS}
\begin{align*}
\langle A\frac{v}{\theta}F_{f^\infty g}, \Pi g\rangle&\leq \frac12\|\frac{v}{\theta}F_{f^\infty g}\|\| \Pi g\|\\
&\leq \frac12\|F_{f^\infty g}\|_{L^\infty}\left\|\frac{v}{\theta}\right\|\| \Pi g\|\\
&\leq \frac{I\mathcal{M}}{2\sqrt{\theta}}\|\pa_x K\|_{L^\infty}\|\Pi g\|^2
\end{align*}
}

{For the last term, we notice that $A\partial_v^* = (\Id +\theta\partial_x^{*}\partial_x)^{-1}(\partial_x\Pi)^*$ and therefore using that the operator norm of $\partial_x(\Id +\theta\partial_x^{*}\partial_x)^{-1}$ is bounded by $1/2$,  
\begin{align*}
\langle AQ(g,g) , \Pi g \rangle &= \langle F_{f^\infty g}\Pi g , \partial_x(\Id +\theta\partial_x^{*}\partial_x)^{-1}\Pi g \rangle\\
&= \frac{1}{2}\| F_{f^\infty g}\Pi g\| \|\Pi g\| \\
&= \frac{1}{2}\|F_{f^\infty g}\|_{L^\infty}\| \Pi g\|^2\\
&= \frac{I}{2}\|\pa_x K\|_{L^\infty}\|\rho_g\|_{L^1_x}\|\Pi g\|^2
\end{align*}
}

Altogether, \eqref{eq:hypoNLH} reads
\begin{align*}
	\frac{\dd}{\dd t} H[g](t) &\leq -2\nu\theta \| \pa_v g\|^2  - \eps \frac{\lambda_M}{1+\lambda_M} \| \Pi g\|^2 + {\eps (1+C_M) \sqrt{\theta}\| \pa_v g\| (\|\Pi g\|+ \sqrt{\theta}\| \pa_v g\|)} \\
	&\quad + {I\mathcal{M}} \|\pa_x K\|_{L^\infty} \| \Pi g \| \| \pa_v g \|  + {\frac{\eps I \mathcal{M}}{2\sqrt{\theta}}} \| \pa_x K\|_{L^\infty} \| \Pi g \|^2 \\
	& \quad + {I \| \pa_x K\|_{L^\infty} \| \rho - \rho^\infty \|_{L^1} \| \Pi g \|\|\partial_v g\|} + {\frac{\eps I }{2}} \| \pa_x K\|_{L^\infty} \| \rho - \rho^\infty \|_{L^1} \| \Pi g \|^2,
\end{align*}
{with $\lambda_M$ and $C_M $} given in Lemma~\ref{lem:LinHypo} {and where we have used the triangle inequality to bound $\|g\|\leq\|\Pi g\| + \|(\Id-\Pi) g\|$ and the Gaussian Poincaré inequality to get $\|(\Id-\Pi)g\|^2\leq \theta \|\partial_vg\|^2$}. {The five first terms define a quadratic form on $(\|\Pi g\|, \sqrt{\theta}\|\partial_vg\|)$ with matrix
\[
M
= \left(\begin{matrix}
-2\nu+\eps(1+C_M) & \ds\frac{\eps(1+C_M)\sqrt{\theta} + I\mathcal{M}\| \pa_x K\|_{L^\infty}}{2\sqrt{\theta}}\\[.75em]
\ds\frac{\eps(1+C_M) \sqrt{\theta}+ I\mathcal{M}\| \pa_x K\|_{L^\infty}}{2\sqrt{\theta}} & \ds-\eps(\frac{\lambda_M}{1+\lambda_M}-\frac{I\mathcal{M}\| \pa_x K\|_{L^\infty}}{2\sqrt{\theta}})\\
\end{matrix}\right) .
\]
By replacing $C_M$ by its expression and introducing $a = \frac{I\mathcal{M}\| \pa_x K\|_{L^\infty}}{\sqrt{\theta}}$, $b = 1+\nu+4\theta^2+4\alpha\theta$ and $c =  \frac{\lambda_M}{1+\lambda_M}$ it rewrites
\begin{equation}\label{eq:matrix_lambda}
M= \left(\begin{matrix}
-2\nu+\eps(b+2\theta a^2) & \ds\frac{\eps(b+2\theta a^2)}{2}+a\\[.75em]
\ds\frac{\eps(b+2\theta a^2)}{2}+a& \ds-\eps c + \eps a\\
\end{matrix}\right) .
\end{equation}
From there, if we assume 
\[
\eps <\frac{2\nu}{b+2\theta}  \quad\text{and}\quad a<\min\{\eps,c\}<1
\]
then the diagonal terms of \eqref{eq:matrix_lambda} are negative hence its trace is negative. Moreover since $a<\eps$ and $c = (1+\lambda_M^{-1})^{-1}> (1+\frac{\alpha}{\theta} \exp({\frac{4\|K\|_{L^\infty}}{\sqrt{\theta}\|\partial_xK\|_{L^\infty}}}))^{-1} := d$ (because $a<1$), one has
\[
\mathrm{det}(M) > \eps(2\nu-\eps(b+2\theta\eps^2))(d-\eps) - \eps^2(\frac b2 +\theta\eps^2+1)^2\sim_{\eps\to0}2\nu d\eps.
\]
Therefore there is $\varepsilon_0>0$ as well as $a_0$ depending only on $\nu,\alpha, \theta,\lambda_M$ such that if $a<a_0$ and $\eps<\eps_0$ then
\begin{align*}
	\frac{\dd}{\dd t} H[g](t) &\leq -\tilde{\lambda}(\theta\| \pa_v g\|^2 + \|\Pi g\|^2) + \frac12\| \Pi g \|_{L^1} (2\sqrt{\theta}a_0\|\partial_v g\| +\sqrt{\theta}\eps_0a_0\| \Pi g \|)\| \Pi g \|
\end{align*}
for some $\tilde{\lambda} >0$ (the opposite of the smallest eigenvalue of the matrix $M$). One concludes with any $\lambda<\tilde{\lambda}$ thanks to Young's inequality and the Gaussian Poincaré inequality.
}

\end{proof}

\subsection{Proof of Theorem \ref{theo:long}}\label{sec:proof_theo_3}

By {Cauchy-Schwarz inequality, one has
\begin{align*} \label{eq:Controlrho}
    \| \Pi g\|_{L^1_x}^2 \leq \| \Pi g \|^2\mathcal{M} \leq \frac{\mathcal{M}}{1-\eps} H[g](t). 
\end{align*}
which in turn proves that for the constants $\lambda, C, \eps_0>0$ of Proposition~\ref{prop:hypocoercivity} and under the same assumptions
\begin{equation}\label{eq:HypoGronwall2} 
		\frac12 \frac{\dd}{\dd t} H[g](t) \leq - \lambda H[g](t) + \frac{C\mathcal{M}}{1-\eps}  (H[g](t))^2,
\end{equation}
with $\eps<\eps_0$. If $H[g](0)< \frac{(1-\eps)\lambda}{C\mathcal{M}}$ then $H[g]$ decays and therefore, for $\lambda^* = \lambda(1-\frac{(1-\eps)}{C\mathcal{M}}) $ one has from \eqref{eq:HypoGronwall}
\[
 H[g](t) \leq  H[g](0) e^{-\lambda^*t}.
\]
Bootstraping this into  \eqref{eq:HypoGronwall2} yields 
\begin{align*} 
		\frac12 \frac{\dd}{\dd t} H[g](t) \leq - \lambda(1-e^{-\lambda^*t}) H[g](t),
\end{align*}}
and Theorem \ref{theo:long} follows.

\section{Derivation of the Vlasov-Fokker-Planck equation}\label{sec:derivation}

The Vlasov-Fokker-Planck equation \eqref{eq:VFPnonlin} arises in the modeling of the longitudinal dynamics of electron bunches in the storage ring of synchrotron particle accelerator. An electron storage ring is used to store ultrarelativistic electron bunches along a closed orbit. This confinement is achieved by various electromagnetic devices. A schematic drawing can be found in \cite[Figure 1.1]{roussel2014spatio}.

 In this section we first provide some elements concerning the derivation of the model, following mainly \cite{roussel2014spatio}. Then we put a particular emphasis on the derivation of the interaction kernel $K$ from Maxwell's equation, following Murphy, Krinsky and Gluckstern \cite{murphy1996longitudinal}. On this part we formalize mathematically some formal arguments of the latter paper in Proposition~\ref{prop:mur} and Proposition~\ref{prop:approx}. \blue{Note that the interaction kernel that we derive in this section is a particular case of the family of kernels considered in the previous sections. To avoid confusion we have decided to adopt different notations for the kernels and potentials in this section.}

\subsection{Particle dynamics}
Consider a charge $e$ traveling  in the storage ring of a particle accelerator. The orbit is assumed to be circular with radius $R_0$ and the velocity to be a fraction $\beta\in(0,1)$ of the speed of light $c$.

The motion of particles is described relatively to a reference orbit. The ideal particle does one turn of the device in $t_0 =2\pi R_0/(\beta c)$ and has a relativistic energy $\mathcal{E}_0$. At turn $n$ in the ring, a particle is referenced with respect to its (dimensionless) relative energy $\delta^n = (\mathcal{E}^n-\mathcal{E}_0)/\mathcal{E}_0$ and its longitudinal position on the orbit $z^n$ in the reference frame of the ideal particle (i.e. $z^n = 0$ for the ideal trajectory). It is assumed that the length and energy spread of the electron bunch is small so that $z^n\ll 1$ and $\delta^n\ll 1$. Because an electron in the bunch has an energy which is slightly different from  the nominal energy, it deviates from the ideal trajectory {which} implies an offset at each turn. More precisely 
\[z^{n+1} = z^n - 2\pi R_0 \eta \delta^n\]
with $\eta$ the slippage factor which is related to the fact that the {length of the orbit} is greater for a particle with higher energy (see \cite{wiedemann2015particle} for details). 
One can write an energy balance to account for energy variation at each turn. It yields \[\mathcal{\delta}^{n+1}\mathcal{E}_0 = \mathcal{\delta}^{n}\mathcal{E}_0 + e V_{\text{rf}}(z^n)  - 2\pi R_0 e E^n_\varphi(z^n) - U(\delta^n) + \mathcal{E}_0\sqrt{2 D}\xi^n.\] On the right-hand side, the second term is related to the acceleration of particles by the RF cavity in the ring. The latter in synced with the period of rotation of the ideal particle and will deliver a potential $V_{\text{rf}}(z^n)$ which varies depending on the {longitudinal displacement $z^n$}. The third term is related to the collective effects due to the self-consistent tangential electric field $E^n_\varphi(z^n)$ created by the bunch. The fourth and fifth terms account for the energy damping due to the emission synchrotron radiation. The term $\xi^n$ is a Gaussian white noise which {models} a stochastic perturbation of the energy loss $- U(\mathcal{E}^n)$, as a result of quantum effects in the emission of photons. Because of the small variation assumption $z^n\ll 1$ and $\delta^n\ll 1$ one has that \[e V_{\text{rf}}(z^n) - U(\delta^n)\approx e V_{\text{rf}}(0)-U(0) + eV_{\text{rf}}'(0)z^n - U'(0)\delta^n.\] The zeroth order term is $ e V_{\text{rf}}(0)-U(0) = 0$, since the RF cavity is tuned to compensate exactly for synchrotron radiation loss of the ideal particle.  Then $V_{\text{rf}}'(0)>0$ is taken to ensure a confinement effect around the nominal energy and $U'(0)>0$.

 After rescaling and non-dimensionalisation of the equations, one can introduce the new variables $x^n = z^n/\sigma_z$ and $v^n = -\delta^n/\sigma_\delta$ with suitably chosen $\sigma_z$ and $\sigma_\delta$. Moreover since the revolution period is much shorter than the typical time of variation of the electron bunch one can and go from discrete $n\in\NN$ to continuous $t\geq0$ number of turn (or up to a constant, time), and deal with continuous in time processes $x_t$ and $v_t$ which will satisfy the Langevin type dynamics
\begin{equation}\label{eq:langevin}
\begin{array}{rcl}
\dd x_t &=& v_t \dd t\\
\dd v_t &=& \underbrace{-\alpha x_t\dd t}_{\text{RF cavity}}\quad\underbrace{+F_f(t,x_t)\dd t}_{\substack{\text{Collective}\\\text{ effect}}}\quad\quad\underbrace{ - 2\nu v_t\dd t +2\sqrt{\nu\theta}\, \dd B_t.}_{\substack{\text{synchrotron radiation}\\\text{losses}}}
\end{array}
\end{equation}
In the equation above the parameters $\alpha, \nu, \theta>0$ are respectively  proportional to the physical quantities $V_{\text{rf}}'(0)$, $U'(0)$ and $D$. The process $(B_t)_{t\geq0}$ is a standard Brownian motion. The collective force term $F_f(t,x_t)$ is proportional to the self consistent tangential electric field which is going to be derived in the following section. The Langevin equations \eqref{eq:langevin} are the particle counterpart of the Vlasov-Fokker-Planck equation \eqref{eq:VFPnonlin}.

A more detailed presentation of the latter derivation can be found in \cite[Section 2.1]{roussel2014spatio} and references therein. We also mention \cite[Lecture 6]{stupakov2007lecture} and \cite{wiedemann2015particle} for additional material on the topic.

\subsection{Wakefield of relativistic particle on a circular orbit}
In this section we present the derivation of the synchrotron radiation reaction force for a relativistic charge, rotating on a circular orbit in free space. The associated field is called the free space wakefield of a point charge. Most of the arguments presented in this section follow from the paper \cite{murphy1996longitudinal} (see also the more recent \cite[Lecture 24]{stupakov2007lecture}). It consists in an asymptotic expansion, in the ultrarelativistic limit, of the longitudinal electromagnetic field created by the particle along its orbit. %At the end of the section we propose a (up to our knowledge) novel formal estimate in Proposition~\ref{prop:wakefield_expansion}, which quantifies the validity of this approximation. 

We assume that a particle of charge $e$ is traveling at position $\mathbf{r}_0(t)$ with velocity $c\boldsymbol{\beta}_0(t)=\mathbf{r}_0'(t)$. Writing $\phi\equiv\phi(t,\mathbf{x})\in\RR$ and $\mathbf{A}\equiv \mathbf{A}(t,\mathbf{x})\in\RR^3$ the resulting scalar and vector potentials at time $t$ and position $\mathbf{x}\in\RR^3$, Maxwell equations in the vacuum with the Lorentz gauge yields in SI units \blue{read
\[ \Box \phi = \eps_0^{-1} e \delta(\mathbf{x}-\mathbf{r}_0(t)), \quad \Box \mathbf{A} = \mu_0c e  \beta_0(t) \delta(\mathbf{x}-\mathbf{r}_0(t))\]}
with  $\Box = c^{-2}\partial^2_{tt} - \Delta_\mathbf{x}$ the d'Alembertian operator. Introducing the fundamental solution of the d'Alembertian operator in dimension $3$, namely	$ \delta\left( |\mathbf{x}| - ct\right)/(4\pi|\mathbf{x}|)$, the solutions to the Maxwell equations are given by the Liénard–Wiechert potentials
	\begin{equation}\label{eq:LWscalar}
		\phi(t,\mathbf{x}) =  \frac{e}{4\pi\eps_0} \frac{1}{|\mathbf{x}-\mathbf{r}_0(\tau)|-\boldsymbol{\beta}_0(\tau)\cdot(\mathbf{x}-\mathbf{r}_0(\tau)) },
	\end{equation}
	and
	\begin{equation}\label{eq:LWvector}
		\mathbf{A}(t,\mathbf{x}) = \frac{ce\mu_0}{4\pi} \frac{\boldsymbol{\beta}_0(\tau)}{|\mathbf{x}-\mathbf{r}_0(\tau)|-\boldsymbol{\beta}_0(\tau)\cdot(\mathbf{x}-\mathbf{r}_0(\tau))}. 
	\end{equation}
	The retarded time $\tau\equiv\tau(t,\mathbf{x})$ is implicitly defined as the unique solution of  
	\[t = \tau + \frac{|\mathbf{x}-\mathbf{r}_0(\tau)|}{c}.\] 
	
	\blue{Next, we express this potentials in terms of angular displacement with respect to the point of observation $\mathbf{x}$. To that end, since} the trajectory of a particle in the storage ring is assumed to be circular with radius $R_0$ constant, \blue{we introduce} the parametrization
	\[ 
	\mathbf{r}_0(t) = R_0 \big(\cos(\omega_0 t+\delta), \sin(\omega_0 t+\delta),0\big)
	\]
	where $\omega_0 = \beta c /R_0$, with constant relative velocity $\beta = |\boldsymbol{\beta}_0|\in[0,1]$, and a constant phase $\delta \in \RR$. We also parametrize the point of observation $\mathbf{x}$ on the orbit as 
	\[\mathbf{x} = ( R_0 \cos \varphi, R_0  \sin \varphi, 0) \]
	and we introduce the (unique) angles $\alpha,\xi\in(0,\pi)$ such that
	\begin{equation*}
		\left\{ \begin{aligned} \alpha(t,x) &\equiv \pi - \frac12 (\omega_0\tau(t,x) + \delta - \varphi)&[\pi], \\
			\xi(t) &\equiv \pi - \frac12 (\omega_0 t + \delta - \varphi)&[\pi]. \end{aligned} \right. 
	\end{equation*} 
	%If we denote the angles between $x$ and $r_0(\tau)$, and $x$ and $r_0(t)$ by $2\chi$ and $2\psi$, then $\psi = \pi - \xi$ and $\chi = \pi - \alpha$. 
	We refer to Figure~\ref{fig:angles} for an illustration.
    	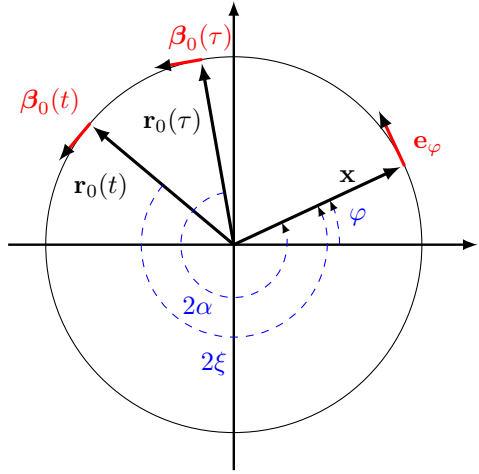
\begin{figure}
    \begin{tikzpicture}
      \def\xmax{3}
      \def\ul{.8}
      \def\R{2.5}
      \def\angx{25}
      \def\angtau{100}
      \def\angt{140}
      \coordinate (O) at (0,0);
      \coordinate (x1) at (\xmax,0);
      \coordinate (x) at (\angx:\R);
      \coordinate (r0tau) at (\angtau:\R);
      \coordinate (r0t) at (\angt:\R);
      \coordinate (xhalf) at (\angx:\R/1.5);
      \coordinate (r0tauhalf) at (\angtau:\R/1.5);
      \coordinate (r0thalf) at (\angt:\R/1.5);
      
      % Circle and axes
      \draw[->,line width=0.9] (-\xmax,0) -- (1.08*\xmax,0);
      \draw[->,line width=0.9] (0,-\xmax) -- (0,1.08*\xmax);
      \draw (O) circle (\R);
      % Circle and axes
      \node[circle,inner sep=0.9] (x') at (x) {};
      \node[above] at (xhalf) {$\mathbf{x}$};
      \draw[vector] (O) -- (x'); %node[midway,left=5,above right=0] {${R0}$};
      \node[circle,inner sep=0.9] (r0tau') at (r0tau) {};
      \node[left] at (r0tauhalf) {$\mathbf{r}_0(\tau)$};
      \draw[vector] (O) -- (r0tau'); %node[midway,left=5,above right=0] {${R0}$};
      \node[circle,inner sep=0.9] (r0t') at (r0t) {};
      \node[below left] at (r0thalf) {$\mathbf{r}_0(t)$};
      \draw[vector] (O) -- (r0t'); %node[midway,left=5,above right=0] {${R0}$};
      \draw pic[->,"$\varphi$",draw=blue, text=blue, dashed, angle radius=40,angle eccentricity=1.2] {angle=x1--O--x};
      %\draw pic[->,"$2\chi$",draw=black,dashed,angle radius=50,angle eccentricity=1.2] {angle=x--O--r0tau};
      %\draw pic[->,"$2\psi\ $",draw=black,dashed,angle radius=30,angle eccentricity=1.2] {angle=x--O--r0t};
      \draw pic[->,"$2\alpha\ $",draw=blue, text=blue,dashed,angle radius=20,angle eccentricity=1.3] {angle=r0tau--O--x};
      \draw pic[->,"$2\xi\ $",draw=blue, text=blue,dashed,angle radius=35,angle eccentricity=1.3] {angle=r0t--O--x};
      \draw[vector,->,line cap=round, draw = red] (x) --++ (\angx+90:\ul);
      \node[above right, text = red] at (x) {$\mathbf{e}_\varphi$};      
       %node[scale=1,above right=-3] {$E_\varphi(t,x) e_\varphi$};
      \draw[vector,->,line cap=round, draw = red] (r0tau) --++ (\angtau+90:.8*\ul);
      \node[above, text = red] at (r0tau) {$\boldsymbol{\beta}_0(\tau)$}; %node[scale=1,above =-3] {$\beta_0(\tau)$};
      \draw[vector,->,line cap=round, draw = red] (r0t) --++ (\angt+90:.8*\ul);
      \node[above left, text = red] at (r0t) {$\boldsymbol{\beta}_0(t)$}; %node[scale=1,left =-3] {$\beta_0(t)$};
    \end{tikzpicture}
    \caption{Parametrization of the particle trajectory.}\label{fig:angles}
    \end{figure}
	We have by definition $|\mathbf{x}-\mathbf{r}_0(\tau)| = 2R_0 \sin(\alpha)$ and since $\boldsymbol{\beta}_0$ is tangential to the circular orbit, $\boldsymbol{\beta}_0(\tau) \cdot (\boldsymbol{x}-\boldsymbol{r}_0(\tau)) = 2R_0\beta \sin(\alpha)\cos(\alpha)$, hence we can rewrite the potentials as
	\begin{equation*}
		\left\{ \begin{aligned} \phi(t,\mathbf{x}) &= \frac{e}{8\pi\eps_0R_0} \frac{1}{\sin\alpha -\beta \sin\alpha\cos\alpha} \\
			A(t,\mathbf{x}) &= \frac{ce\mu_0}{8\pi R_0} \frac{\boldsymbol{\beta}_0(\tau)}{\sin\alpha -\beta \sin\alpha\cos\alpha}.,
		\end{aligned} \right. 
	\end{equation*}
	and the relation between the angle $\xi$ and the retarded angle $\alpha$ is
	\[ 
	\xi = \alpha - \beta \sin \alpha.
	\]
	The limit $\xi,\alpha\to0^+$ and  $\xi,\alpha\to\pi^-$ correspond respectively to the potentials created just in front and just behind the particle, on the circular orbit.  Let us now express the tangential electric field at the observation point, namely 
	\[ E_\varphi (t,\mathbf{x}) = -\frac{1}{R_0} \frac{\pa \phi}{\pa \varphi} - \frac{\pa \mathbf{A}\cdot\mathbf{e}_\varphi}{\pa t}. \]
	%where the tangential component of the vector potential $A$ at $x$ is 
	%\[
	%A_\varphi = \frac{ce\mu_0}{8\pi R_0} \frac{\beta\cos(2\alpha)}{\sin\alpha -\beta \sin\alpha\cos\alpha}
	%\]
By definition of the angles, one has the following relations
	\[\frac{\pa}{\pa \varphi} =  \frac12 \frac{\pa}{\pa \xi} \quad\text{and}\quad
			\frac{\pa}{\pa t} = -\frac{c\beta}{2R_0} \frac{\pa}{\pa \xi}.\]
	As a result,  using also that $\varepsilon_0\mu_0c^2 = 1$,  we get
	\[
		E_\varphi  %= -\frac{\pa}{\pa \xi} \left[ \frac{\phi}{2R_0}  - \frac{c\beta A_\varphi}{2R_0} \right] 
		= -\frac{e}{8\pi\eps_0R_0^2}\frac{\pa}{\pa \xi} \left[ \frac{1-\beta^2(1-2\sin^2\alpha)}{2(\sin\alpha-\beta\sin\alpha\cos\alpha)} \right].
	\]
	We extend the previous formula to the negative angles $\xi,\alpha\in(-\pi,0)$ using the $\pi$ periodicity of the parametrization. It yields
	\begin{equation}\label{eq:tang_elec}
	E_\varphi = -\frac{e}{8\pi\eps_0R_0^2}\frac{\partial V}{\partial \xi}
    \end{equation}	
	with a dimensionless potential $V$ given by
	\begin{align}
	V(\xi) = \frac{1-\beta^2(1-2\sin^2\alpha)}{2(|\sin\alpha|-\beta\sin\alpha\cos\alpha)} \,,\\
	\xi = \alpha-\beta|\sin\alpha|\label{eq:relalphaxi}\,,
	\end{align}
	 for $\xi,\alpha\in(-\pi,\pi)$.	 Let us now decompose the potential into two parts
	 \[
	 V(\xi) = V^\text{C}(\xi) + V^\text{S}(\xi).
	 \]
	 The first part corresponds to the singular Coulomb part of the potential defined by 
	 \begin{equation}\label{eq:Coulpot}
	 V^\text{C}(\xi) = \frac{1-\beta^2}{2|\sin\xi|}\,.
	 \end{equation}
	 The second part is related to the synchrotron radiation reaction force and reads
	 \begin{equation}\label{eq:Syncpot}
	 V^\text{S}(\xi) = \frac{(1-\beta^2)(|\sin\xi|-|\sin\alpha|+\beta\sin\alpha\cos\alpha) + 2\beta^2\sin^2\alpha|\sin\xi|}{2|\sin\xi|(|\sin\alpha|-\beta\sin\alpha\cos\alpha)}.
	 \end{equation}
	 
	 \subsection{\blue{Ultra-relativistic limit of the wakefield}} 
	 \blue{We are interested in the behaviour of $V^\text{C}(\xi)$ and $V^\text{S}(\xi)$ in the ultra-relativisitic limit, i.e. when the speed of the particle tends to the speed of light. In the particle physics literature, it is argued that in this limit the Coulomb part $V^\text{C}(\xi)$ is negligible compared to the synchrotron radiation reaction part $V^\text{S}(\xi)$. The purpose of this section is give some elements of mathematical justification for this negligibility. \\
	 We will characterize the ultra-relativistic regime is terms of the Lorentz factor $\gamma = (1-\beta^2)^{-\frac12}$ which goes to $+\infty$. Using \eqref{eq:relalphaxi} one finds that the Taylor expansion of $\alpha$ can be written as 
	 	\[
	\alpha = 
	\left\{
	\begin{aligned}
	(\Omega^{\frac13} - \Omega^{-\frac13})\gamma^{-1} + \frac{\Omega^{\frac53}+5\Omega-35\Omega^{\frac13}+35\Omega^{-\frac13}- 5\Omega^{-1}-\Omega^{-\frac53}}{60(\Omega^{\frac23} + \Omega^{-\frac23}-1)}\gamma^{-3} + O(\gamma^{-5})&\text{ if }\mu>0\,, \\
	\frac{\mu}{6}\gamma^{-3} + O(\gamma^{-5})&\text{ if }\mu<0\,.
	\end{aligned}
	\right.
	\]
	where $\Omega = \mu+\sqrt{\mu^2+1}$ and $\mu = 3\gamma^3\xi$  which motivates the following equivalent of $V^\text{S}$. This formula is due to Murphy, Krinsky and Gluckstern \cite{murphy1996longitudinal}.}
	 \begin{prop}\label{prop:mur}
	 Let $\gamma = (1-\beta^2)^{-\frac12}$ be the Lorentz factor and let 
	 \begin{equation}\label{eq:rescaling}
	 \mu = 3\gamma^3\xi
	 \end{equation}
	  be a given non-zero real number. Then in the ultra-relativistic limit $\gamma\to+\infty$ one has the asymptotic expansion
	 \begin{equation}\label{eq:Syncpot_expansion}
	 V^\text{S}\left(\xi\right)  = \gamma K^\text{fs}(\mu) + \left\{
	\begin{aligned} O(\gamma^{-1})&\ \text{for}\ \mu>0\,,\\
	O(\gamma^{-3})&\ \text{for}\ \mu<0\,.
	\end{aligned}
	\right.
	 \end{equation}
	 with 
\begin{equation}\label{eq:formula_freespace}
K^\text{fs}(\mu) = \left\{
	\begin{aligned} &2\frac{\cosh\left[\frac53\sinh^{-1}\mu\right]-\cosh\left[\sinh^{-1}\mu\right]}{\sinh\left[2\sinh^{-1}\mu\right]}&\text{ if }\mu>0\,, \\
	&0&\text{ if }\mu<0\,.
    \end{aligned}
    \right.
\end{equation}
	 \end{prop}
	 \begin{proof}
	 Direct corollary of the expansion $\alpha$ and the expression \eqref{eq:Syncpot} for $V^\text{S}$. 
	 \end{proof}

\begin{rema} Observe that because of \eqref{eq:rescaling} the angle $\xi$ behaves like $O(\gamma^{-3})$ and therefore, this approximation is only valid for small angles. When one considers collective dynamics with many interacting electrons on a given circular orbit, formula \eqref{eq:formula_freespace} will be valid to describe collective interactions only for a bunch  with small longitudinal spread. A striking fact following from \eqref{eq:Syncpot_expansion}-\eqref{eq:formula_freespace} is that at principal order the wakefield  created by a particle is non-zero only in front of it.
\end{rema}	 
\begin{rema}\label{rem:Coulpot_expansion}
It is immediate to perform the same expansion as in  Proposition~\ref{prop:mur} for the Coulomb part of the potential, and in this case one has that as $\gamma\to\infty$, 
\[
V^\text{C}\left(\xi\right) = \frac{3\gamma}{2\mu} + O(\gamma^{-5}).
\]
Observe that unlike for $V^\text{S}$ (see Figure~\ref{fig:wakefield}), the equivalent of $V^\text{C}$ is singular at the origin.
\end{rema}

\begin{figure}
\centering
\begin{tikzpicture}[scale = 1,
declare function={
Vfs(\x) = 2 * (cosh((5/3)*asinh(\x)) - cosh(asinh(\x))) / sinh(2*asinh(\x))*(\x>0);
}, 
declare function={
dVfs(\x) =(\x>0.1)*(- (6*\x*(\x + (\x^2 + 1)^(1/2))^(5/3) - 5*(\x + (\x^2 + 1)^(1/2))^(10/3) + 5)/(6*\x*(\x^2 + 1)*(\x + (\x^2 + 1)^(1/2))^(5/3)) - ((2*\x^2 + 1)*(1/(\x + (\x^2 + 1)^(1/2))^(5/3) - 2*(\x^2 + 1)^(1/2) + (\x + (\x^2 + 1)^(1/2))^(5/3)))/(2*(\x^4 + \x^2)*(\x^2 + 1)^(1/2))) + (\x>0)*(\x<=0.1)*((11200*\x^4)/6561 - (112*\x^2)/81 + 8/9);% Taylor expansion at 0 to avoid numerical error in tikz computation. Computed thanks to matlab symbolic toolbox
}
]
\begin{axis}[
    xlabel=$\mu$,
    samples=200, % Number of samples for the function plot
    axis lines=middle, % Draw the axes
    xmin=-2, xmax=8.5, % Set the limits for the x-axis
    ymin=-.1, ymax=1.1, % Set the limits for the y-axis
    clip=false, % Allow functions to be plotted outside of the axis limits
]
%% Plot the functions
\addplot[domain=-2:8.5, blue, sharp plot] {Vfs(x)};
\addplot[domain=0.0001:8.5, red, dashed, sharp plot] {dVfs(x)};%% Add the legend
\legend{$K^\text{fs}(\mu)$, $\frac{\dd K^\text{fs}(\mu)}{\dd\mu}$}
\end{axis}
\end{tikzpicture}
\caption{Free space wakefield potentials}\label{fig:wakefield}
\end{figure}
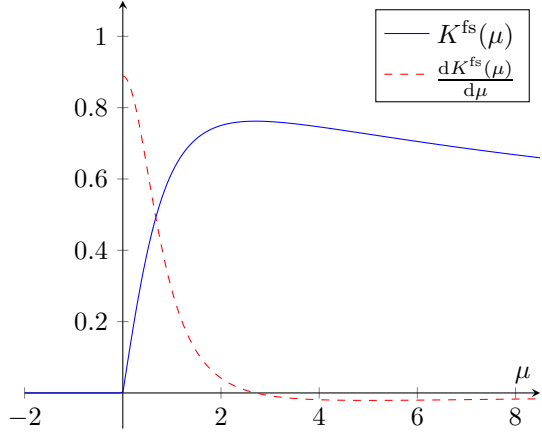

On the scale of $\xi$, one can indeed check that $V^\text{C}(\xi)$ is $O(\gamma^{-2})$ compared to $V^\text{S}(\xi)$. However on the scale of $\mu$ (when $\xi = O(\gamma^{-3})$) where the approximation $\gamma K^\text{fs}(3\gamma^3\xi)$ is derived, the Coulomb term is not negligible anymore (see Remark~\ref{rem:Coulpot_expansion}). In the next Proposition we show that $\gamma K^\text{fs}(3\gamma^3\xi)$ is indeed a good approximation of the potential $V = V^C+V^S$ in front of the charge (\emph{i.e.} $\xi>0$) on the scale 
\[
\gamma^{-3}\ll \xi\ll 1\,.
\]
Behind the charge (\emph{i.e.} $\xi<0$), it can be assumed that the potential vanishes (and therefore $\gamma K^\text{fs}(3\gamma^3\xi)$ is again a suitable approximation) on the scale
\[
\gamma^{-2}\ll -\xi\ll 1.
\]
More precisely one has the following result.
\begin{prop}\label{prop:approx}
Let $\xi = \xi_0 \varepsilon(\gamma)$ with $\varepsilon$ a continuous positive   function of the Lorentz factor $\gamma$ such that $\varepsilon(\gamma) = o(1)$ and $\gamma^{-3} = o(\varepsilon(\gamma))$ as $\gamma\to\infty$. Then
\begin{equation}
V(\xi)\sim \gamma K^\text{fs}(3\gamma^3\xi)\quad\text{as }\gamma\to\infty\,,\quad\text{if }\xi_0>0.
\end{equation}
If additionally $\gamma^{-2} = o(\varepsilon(\gamma))$, then
\begin{equation}
\lim_{\gamma\to\infty}V(\xi) = 0,\quad\text{if }\xi_0<0.
\end{equation}
\end{prop}
\begin{proof} Since $\xi(\gamma) = \alpha(\gamma) - \sqrt{1-\gamma^{-2}} |\sin \alpha(\gamma)|$ and $\xi(\gamma)\to0$ as $\gamma\to \infty$, $\alpha(\gamma)$ also tends to $0$ as $\gamma\to \infty$. 

In the case $\xi_0>0$, an asymptotic expansion of \eqref{eq:relalphaxi} yields that 
$
\xi = \frac12\gamma^{-2}\alpha+ \frac16\alpha^3 + o(\gamma^{-2}\alpha+ \alpha^3)
$. Since $\gamma^{-3} = o(\xi)$, one shows that $\gamma^{-2}\alpha = o(\alpha^3)$ and therefore
\[
\xi =\frac{\alpha^3}{6} + o(\alpha^3)\,, \quad\text{if }\xi_0>0.
\]
In particular $\gamma^{-1} = o(\alpha)$. Inserting this expansion into the expression \eqref{eq:Syncpot} one finds
\[
V^\text{S}(\xi)\sim\frac{2^{\frac23}}{(3\xi)^{\frac13}}\,, \quad\text{if }\xi_0>0.
\]
One finds the same equivalent for $
\gamma K^\text{fs}(3\gamma^3\xi)$ as $\gamma\to\infty$.

In the case $\xi_0<0$, an asymptotic expansion of \eqref{eq:relalphaxi} shows that $\gamma^{-1} = o(|\alpha|^{\frac13})$ and 
\[
\xi = 2\alpha - \frac12 \gamma^{-2}\alpha -\frac{1}{8}\gamma^{-4}\alpha-\frac{\alpha^3}{6} +  o(|\alpha|^3)\,, \quad\text{if }\xi_0<0.
\]
Inserting this expansion into the expression \eqref{eq:Syncpot} one finds
\[
V^\text{S}(\xi)\sim -\frac{\xi^2}{8}\,, \quad\text{if }\xi_0<0.
\]

Finally \[V^\text{C}(\xi)\sim \frac{\gamma^{-2}}{|\xi|}\] as $\gamma\to\infty$ for $\xi_0\neq0$. The combination of the previous asymptotic expansions and the hypotheses $\gamma^{-3} = o(\xi)$ for $\xi_0>0$ and $\gamma^{-2} = o(\xi)$ for $\xi_0<0$ allow to prove the claims.
\end{proof}

\begin{rema}(Other wakefields)
The interaction potential $K^\text{fs}$ is too idealized to model an actual storage rings where interactions between the electromagnetic field and the boundaries of the vacuum chamber are non negligible. The computations above can be adapted to take into account these effects, for instance using the parallel plate wakefield where the circular orbit is assumed to be between two infinite conductive plates. In that case computations can still be carried on and the resulting potential is typically of the form \[K = K^\text{fs}+G,\] up to rescalings. The function $G$ is  bounded and smooth and encodes the effect of the reflections of the fields on the boundary. Unlike $K^\text{fs}$, $G$ is typically not supported only on the half line meaning that electron can interact with an electron behind it, thanks to reflected fields. For more details we refer to \cite{murphy1996longitudinal, roussel2014spatio}.
\end{rema}

\bibliographystyle{plain}
\bibliography{bibli}

\end{document}